\documentclass[10]{article}
\usepackage{amsfonts,amssymb,amsmath,amsthm,cite}
\usepackage{graphicx}

\usepackage[applemac]{inputenc}
\usepackage{amsmath,amssymb,amsthm, hyperref, euscript}
\usepackage[matrix,arrow,curve]{xy}
\usepackage{graphicx}
\usepackage{tabularx}
\usepackage{float}
\usepackage{hyperref}
\usepackage{tikz}
\usepackage{slashed}

\usetikzlibrary{matrix}

\usepackage[T1]{fontenc}
\usepackage{amsfonts,cite}
\usepackage{graphicx}

\usepackage[applemac]{inputenc}

\usepackage[sc]{mathpazo}
\DeclareFontFamily{T1}{pzc}{}
\DeclareFontShape{T1}{pzc}{m}{it}{1.8 <-> pzcmi8t}{}
\DeclareMathAlphabet{\mathpzc}{T1}{pzc}{m}{it}

\theoremstyle{plain}
\newtheorem{prop}{Proposition}[section]

\newtheorem{lem}[prop]{Lemma}
\newtheorem{thm}[prop]{Theorem}

\theoremstyle{definition}
\newtheorem{defn}[prop]{Definition}
\newtheorem{empt}[prop]{}
\newtheorem{exm}[prop]{Example}
\newtheorem{rem}[prop]{Remark}

\DeclareMathOperator{\Dom}{Dom}              

\newcommand{\vertiii}[1]{{\left\vert\kern-0.25ex\left\vert\kern-0.25ex\left\vert #1
    \right\vert\kern-0.25ex\right\vert\kern-0.25ex\right\vert}}
\newcommand{\Ga}{\Gamma}                     
\newcommand{\Coo}{C^\infty}                  

\usepackage[sc]{mathpazo}
\newbox\ncintdbox \newbox\ncinttbox 
\setbox0=\hbox{$-$} \setbox2=\hbox{$\displaystyle\int$}
\setbox\ncintdbox=\hbox{\rlap{\hbox
    to \wd2{\hskip-.125em \box2\relax\hfil}}\box0\kern.1em}
\setbox0=\hbox{$\vcenter{\hrule width 4pt}$}
\setbox2=\hbox{$\textstyle\int$} \setbox\ncinttbox=\hbox{\rlap{\hbox
    to \wd2{\hskip-.175em \box2\relax\hfil}}\box0\kern.1em}

\newcommand{\ncint}{\mathop{\mathchoice{\copy\ncintdbox}%
           {\copy\ncinttbox}{\copy\ncinttbox}%
           {\copy\ncinttbox}}\nolimits}  

\title{Infinite  Noncommutative Covering Projections}
\begin{document}
\maketitle  \setlength{\parindent}{0pt}
\begin{center}
\author{
{\textbf{Petr R. Ivankov*}\\
e-mail: * monster.ivankov@gmail.com \\
}
}
\end{center}

\vspace{1 in}

\noindent

Gelfand - Na\u{i}mark theorem supplies a one to one correspondence between commutative $C^*$-algebras and locally compact Hausdorff spaces. So any noncommutative $C^*$-algebra can be regarded as a generalization of a topological space.  Generalizations of several topological invariants may be defined by algebraic methods. For example Serre Swan theorem \cite{karoubi:k} states that complex topological $K$-theory coincides with $K$-theory of $C^*$-algebras. This article devoted to the noncommutative generalization of infinite covering projections. Infinite covering projections of spectral triples are also discussed. It is shown that covering projection of foliation algebras can be constructed by topological coverings of foliations and isospectral deformations. Described an interrelationship between noncommutative covering projections and $K$-homology. The Dixmier trace of noncommutative covering projections is discussed.

\tableofcontents

\section{Motivation. Preliminaries}

Following Gelfand-Na\u{i}mark theorem  \cite{arveson:c_alg_invt} states the correspondence between  locally compact Hausdorff topological spaces and commutative $C^*$-algebras.

\begin{thm}\label{gelfand-naimark}\cite{arveson:c_alg_invt}
Let $A$ be a commutative $C^*$-algebra and let $\mathcal{X}$ be the spectrum of A. There is the natural $*$-isomorphism $\gamma:A \to C_0(\mathcal{X})$.
\end{thm}

So any (noncommutative) $C^*$-algebra may be regarded as a generalized (noncommutative)  locally compact Hausdorff topological space. But $*$-homomorphisms are not good analogs of continuous maps, because there is no a $*$-homomorphism $\varphi$ such that $\varphi$ corresponds to a map from a non-compact topological space to a compact one. However there are infinitely listed covering projections $\widetilde{\mathcal{X}} \to \mathcal{X}$ such that $\widetilde{\mathcal{X}}$ (resp. $\mathcal{X}$) is non-compact (resp. compact). A good analog of a continuous map is a $C^*$-correspondence (See definition \ref{corr_defn}). Let us recall several well known facts.
\begin{defn}\cite{spanier:at}
A fibration $p: \mathcal{\widetilde{X}} \to \mathcal{X}$ with unique path lifting is said to be  {\it regular} if, given any closed path $\omega$ in $\mathcal{X}$, either every lifting of $\omega$ is closed or none is closed.
\end{defn}
\begin{defn}\label{cov_proj_cov_grp}\cite{spanier:at}
Let $p: \mathcal{\widetilde{X}} \to \mathcal{X}$ be a covering projection.  A self-equivalence is a homeomorphism $f:\mathcal{\widetilde{X}}\to\mathcal{\widetilde{X}}$ such that $p \circ f = p$). We denote this group by $G(\mathcal{\widetilde{X}}|\mathcal{X})$. This group is said to be the {\it group of covering transformations} of $p$.
\end{defn}
\begin{prop}\cite{spanier:at}
If $p: \mathcal{\widetilde{X}} \to \mathcal{X}$ is a regular covering projection and $\mathcal{\widetilde{X}}$ is connected and locally path connected, then $\mathcal{X}$ is homeomorphic to space of orbits of $G(\mathcal{\widetilde{X}}|\mathcal{X})$, i.e. $\mathcal{X} \approx \mathcal{\widetilde{X}}/G(\mathcal{\widetilde{X}}|\mathcal{X})$. So $p$ is a principal bundle.
\end{prop}
We would like generalize regular covering projections which are principal bundles.
However any principal bundle with a compact group corresponds to a Hopf-Galois extension (See \cite{hajac:toknotes}).
We may summarize several properties of the Gelfand - Na\u{i}mark correspondence with the
following dictionary.
\newline
\break
\begin{tabular}{|c|c|}
\hline
TOPOLOGY & ALGEBRA\\
\hline
Locally compact space & $C^*$ - algebra\\
Compact space & Unital $C^*$ - algebra\\
Continuous map & $C^*$-correspondence\\
Principal bundle with compact group  & Hopf-Galois extension \\
Infinite covering projection  & ? \\
\hline
\end{tabular}
\newline
\newline
\break
Above table contains all ingredients for construction of infinite covering projections
\begin{itemize}
\item Principal bundles,
\item Maps from non-compact spaces to compact ones.
\end{itemize}
However above table does not have  principal bundles with not-compact groups. We shall construct it with an application of von Neumann algebras.
\newline
This article assumes elementary knowledge of following subjects:
\begin{enumerate}
\item Set theory \cite{halmos:set}.
\item Category theory  \cite{spanier:at},
\item Algebraic topology  \cite{spanier:at},
\item $C^*$-algebras and operator theory \cite{pedersen:ca_aut}.
\end{enumerate}

The terms "set", "family" and "collection" are synonyms.
Following table contains used in this paper notations.
\newline
\begin{tabular}{|c|c|}
\hline
Symbol & Meaning\\
\hline
$A^+$  & Unitization of $C^*-$ algebra $A$\\
$A_+$  & A positive cone of $C^*-$ algebra $A$\\
$A''$ & Bicommutant of $C^*$ algebra $A$ \cite{pedersen:ca_aut}\\
$A^G$  & Algebra of $G$ invariants, i.e. $A^G = \{a\in A \ | \ ga=a, \forall g\in G\}$\\
$\hat A$ & Spectrum of  $C^*$ - algebra $A$  with the hull-kernel topology \\
 & (or Jacobson topology)\\
$\mathrm{Aut}(A)$ & Group * - automorphisms of $C^*$  algebra $A$\\
$B(H)$ & Algebra of bounded operators on Hilbert space $H$\\
$B_{\infty}=B_{\infty}(\{z\in \mathbb{C} \ | \ |z|=1\})$  & Algebra of Borel measured functions on the $\{z\in \mathbb{C} \ | \ |z|=1\}$ set. \\
$\mathbb{C}$ (resp. $\mathbb{R}$)  & Field of complex (resp. real) numbers \\
$\mathbb{C}^*$ & $\{z \in \mathbb{C} \ | \ |z| = 1\}$ \\
$C(\mathcal{X})$ & $C^*$ - algebra of continuous complex valued \\
 & functions on topological space $\mathcal{X}$\\
$C_0(\mathcal{X})$ & $C^*$ - algebra of continuous complex valued \\
 & functions on topological space $\mathcal{X}$\\
 $C_b(\mathcal{X})$ & $C^*$ - algebra of bounded  continuous complex valued \\
  & functions on topological space $\mathcal{X}$\\
$H$ &Hilbert space \\
$I = [0, 1] \subset \mathbb{R}$ & Closed unit  interval\\
$G_{tors} \subset G$  & The torsion subgroup of an abelian group\\
$K(A)$ & Pedersen ideal of $C^*$-algebra $A$\\
$\mathcal{K}(H)$ or $\mathcal{K}$ & Algebra of compact operators on Hilbert space $H$\\
$\mathbb{M}_n(A)$  & The $n \times n$ matrix algebra over $C^*-$ algebra $A$\\

$\mathrm{Map}(X, Y)$  & The set of maps from $X$ to $Y$\\
$M(A)$  & A multiplier algebra of $C^*$-algebra $A$\\
$M^s(A) = M(A \otimes \mathcal{K})$  & Stable multiplier algebra of $C^*-$ algebra $A$\\
$\mathbb{N}$ & Monoid of natural numbers \\
$Q(A)=M(A)/A$  & Outer multiplier algebra of $C^*-$ algebra $A$\\

$Q^s(A)=(M(A \otimes \mathcal{K}))/(A  \otimes \mathcal{K})$  & Stable outer multiplier algebra of $C^*-$ algebra $A$\\

$\mathbb{Q}$  & Field of rational numbers \\
  $\mathrm{sp}(a)$ & Spectrum of element of $C^*$-algebra $a\in A$  \\
  $\mathrm{supp}(f)$ & Support of $f\in C_0(\mathcal{X})$, $\mathrm{supp}(f) = \left\{x \in \mathcal{X} \ | \ f(x)\neq 0 \right\}$   \\
$U(H) \subset \mathcal{B}(H) $ & Group of unitary operators on Hilbert space $H$\\
$U(A) \subset A $ & Group of unitary operators of algebra $A$\\

$\mathbb{Z}$ & Ring of integers \\

$\mathbb{Z}_m$ & Ring of integers modulo $m$ \\
$\Omega$ &  Natural contravariant functor from category  of commutative \\ & $C^*$ - algebras, to category of Hausdorff spaces\\

\hline
\end{tabular}

\section{Prototype. Hopf-Galois extensions}\label{hopf-galois_extensions_section}
The  Hopf-Galois theory supplies a good noncommutative generalization of finite covering projections. Let us recall some notions of Hopf-Galois theory. Following subsection is in fact a citation of \cite{morita_hopf_galois}.
\subsection{Coaction of Hopf algebras}\label{hopf-galois_coactions_section}
\begin{defn}
\label{cat_equivalence_definition}
 An \emph{equivalence of categories} $\mathbf{A}$ and $\mathbf{B}$ is a pair $(F, G)$ of functors ($F: \mathbf{A} \rightarrow \mathbf{B}$, $G: \mathbf{B} \rightarrow \mathbf{A}$) and a pair of natural isomorphisms
 \begin{equation}\nonumber
\alpha : 1_{\mathbf{A}} \to GF,\quad \beta: 1_{\mathbf{B}}\to FG.
 \end{equation}
\end{defn}
Let $ H$ be a Hopf algebra over the commutative ring $\mathbb{C}$, with
bijective antipode $S$. We use the Sweedler notation \cite{karaali:ha} for the comultiplication
on $H: \Delta (h) = h_{(1)} \otimes h_{(2)}$. $\mathcal{M}^H$ (respectively $^H\mathcal{M}$) is the category of right
(respectively left) $H$-comodules. For a right $H$-coaction $\rho$ (respectively a
left $H$-coaction $\lambda$) on a $\mathbb{C}$ -module $M$, we denote
\begin{equation}\nonumber
\rho(m) = m_{[0]}\otimes m_{[1]} ; \  \lambda(m) = m_{[−1]} \otimes m_{[0]}.
\end{equation}
The submodule of coinvariants $M^{\mathrm{co}H}$ of a right (respectively left) $H$-comodule $M$
consists of the elements $m \in M$ satisfying
\begin{equation}\label{right_inv}
\rho(m) = m\otimes 1
\end{equation}
 respectively
\begin{equation}\label{left_inv}
\lambda(m) = 1 \otimes m.
\end{equation}

\begin{defn}\cite{morita_hopf_galois}
Let $A$ be associative algebra and $A\in \mathcal{M}^H$. Algebra $A$ is said to be $H${\it-comodule algebra} if $H$ - coaction $\rho: A \rightarrow A \otimes H$ satisfies following conditions:

\begin{equation}\label{h_action_homo}
\rho(ab) = a_{[0]}b_{[0]} \otimes a_{[1]}b_{[1]}; \ \forall a, b \in A;
\end{equation}
\begin{equation}\label{h_coaction_homo}
a \otimes \Delta(h) = \rho(a) \otimes h.
\end{equation}

\end{defn}
Let $A$ be a right $H$-comodule algebra. $_A\mathcal{M}^H$ and $\mathcal{M}^H_A$
are the categories of
left and right Hopf modules. We have two pairs of adjoint functors
($F_1 = A \otimes_{A^{\mathrm{co}H}} -, G_1 = (-)^{\mathrm{co}H}$) and $(F_2 = − \otimes_{A^{\mathrm{co}H}} A, G_2 = (-)^{\mathrm{co}H}$)
between the categories $_{A^{\mathrm{co}H}}M$ and $_A\mathcal{M}^H$, and between $\mathcal{M}_{A^{\mathrm{co}H}}$ and $\mathcal{M}^H_A$.
The unit and counit of the adjunction $(F_1,G_1)$are given by the formulas
\begin{equation}\nonumber
\eta_{1,N} : N \rightarrow (A \otimes_{A^{\mathrm{co}H}} N)^{\mathrm{co}H}, \ \eta_{1,N}(n) = 1 \otimes  n;
\end{equation}
\begin{equation}\nonumber
\varepsilon_{1,M} : A \otimes_{A^{\mathrm{co}H}} M^{\mathrm{co}H} \rightarrow M, \ \varepsilon_{1,M}(a\otimes m) = am.
\end{equation}

The formulas for the unit and counit of $(F_2,G_2)$ are similar. Consider the
canonical maps
\begin{equation}\label{can_def}
\mathrm{can} : A \otimes_{A^{\mathrm{co}H}} A \rightarrow A \otimes H, \ \mathrm{can}(a \otimes b) = ab_{[0]} \otimes b_{[1]};
\end{equation}\label{can'_def}
\begin{equation}\label{canp_def}
\mathrm{can}' : A \otimes_{A^{\mathrm{co}H}} A\rightarrow A\otimes H, \ \mathrm{can}'(a \otimes b) = a_{[0]}b \otimes a_{[1]}.
\end{equation}

\begin{thm}\label{hoph_galois_def_thm}\cite{morita_hopf_galois}
Let $A$ be a right $H$-comodule algebra. Consider the following
statements:
\begin{enumerate}
\item $(F_2,G_2)$ is a pair of inverse equivalences;
\item  $(F_2,G_2)$ is a pair of inverse equivalences and  $A \in _{A^{\mathrm{co}H}} \mathcal{M}$ is flat;
\item  $\mathrm{can}$ is an isomorphism and $A \in_{A^{\mathrm{co}H}} \mathcal{M}$ is faithfully flat;
\item  $(F_1,G_1)$ is a pair of inverse equivalences;
\item $(F_1,G_1)$ is a pair of inverse equivalences and  $A \in\mathcal{M}_{A^{\mathrm{co}H}}$ is flat;
\item $\mathrm{can}'$ is an isomorphism and $A\in \mathcal{M}_{A^{\mathrm{co}H}}$ is faithfully flat.
\end{enumerate}
These the six conditions are equivalent.
\end{thm}

\begin{defn}\label{hoph_galois_def}
If conditions of  theorem \ref{hopf-galois_coactions_section} are hold, then $A$ is said to be {\it left
faithfully flat $H$-Galois extension}.
\end{defn}

It is well-known  that  $\mathrm{can}$ is an isomorphism if and
only if $\mathrm{can}'$ is an isomorphism.

\subsection{Action of finite group}

Let $G$ be a finite group. A set $H = \mathrm{Map}(G, \mathbb{C})$ has a natural structure of commutative Hopf algebra
(See \cite{hajac:toknotes}). Addition (resp. multiplication) on $H$ is pointwise addition (resp. pointwise
multiplication). Let $\delta_g\in H, ( g \in G)$ be such that
\begin{equation}\label{group_hopf_action_rel}
\delta_g(g')\left\{
\begin{array}{c l}
    1 & g'=g\\
   0 & g' \ne g
\end{array}\right.
\end{equation}

Comultiplication  $\Delta: H \rightarrow H \otimes H$ is induced by group multiplication
\begin{equation}\nonumber
\Delta f(g) = \sum_{g_1 g_2 = g} f(g_1) \otimes f(g_2); \ \forall f \in \mathrm{Map}(G, \mathbb{C}), \ \forall g\in G.
\end{equation}
i.e.
\begin{equation}\nonumber
\Delta \delta_g = \sum_{g_1 g_2 = g} \delta_{g_1} \otimes \delta_{g_2}; \  \forall g\in G.
\end{equation}

Action $G \times A \rightarrow A, \  (g, a) \mapsto ga$  naturally induces coaction $A\rightarrow A \otimes H$ ($H =   \mathrm{Map}(G, \mathbb{C})$) .
\begin{equation}\label{gh_action}
a \mapsto \sum_{g \in G} ga \otimes \delta_g
\end{equation}
Equations (\ref{h_action_homo}), (\ref{h_coaction_homo}) are equivalent to  following conditions of group action
\begin{equation}\nonumber
g(a_1a_2) = (ga_1)(ga_2), \ \forall g \in G, \ a_1, a_2 \in A,
\end{equation}
\begin{equation}\nonumber
(g_1g_2)a = g_1(g_2a) , \ \forall g_1, g_2 \in G, \ a \in A.
\end{equation}
Any element $x \in A\otimes H$ can be represented as following sum
\begin{equation}\nonumber
x = \left(\sum_{g\in G} a_g \otimes \delta_g \right).
\end{equation}

Let $a \in A$ be such that $ga= a, \ \forall g\in G$ then
\begin{equation}\label{g_h_inv_equivalence}
a \mapsto \sum_{g\in G} a \otimes \delta_g = a \otimes 1.
\end{equation}

From (\ref{g_h_inv_equivalence}) it follows that $A^{\mathrm{co}H}=A^G$, where $A^G=\{a \in A: ga= a; \ \forall g\in G \}$ is an algebra of invariants. There is a bijective natural map
\begin{equation}\label{map_tensor_bijection}
A\otimes H \stackrel{\approx}{\longrightarrow} \mathrm{Map}(G, A)
\end{equation}
\begin{equation}\nonumber
\sum_{g \in G} a_g \otimes \delta_g \mapsto \left(g \mapsto a_g\right).
\end{equation}

From (\ref{g_h_inv_equivalence}) it follows that (\ref{can_def})  can be represented in terms of group action by following way

\begin{equation}\label{can_g_def}
\mathrm{can}\left(\sum_{i=1,..., n} a_i \otimes b_i \right) = \sum_{\substack{i=1,..., n\\
g\in G}} a_i(gb_i) \otimes \delta_g.
\end{equation}
There is the unique map  $\mathrm{can}^G: A \otimes_{A^G} A \rightarrow\mathrm{Map}(G, A)$
\begin{equation}\label{canonical_map_defn_eqnhp}
\sum_{i =1, ..., n} a_i\otimes b_i \mapsto (g \mapsto \sum_{i =1, ..., n}a_i (gb_i)), \ \ (a_i,b_i \in B, \ \forall g\in G)
\end{equation}
From bijection of (\ref{map_tensor_bijection}) it follows that  $\mathrm{can}$ is bijective if and only is  $\mathrm{can}^G$ is bijective, i.e.
\begin{equation}\label{finite_g_bijection_equiv}
A \otimes_{A^G} A \approx \mathrm{Map}(G, A).
\end{equation}

Following lemma is an analogue of result described in \cite{miyashita_fin_outer_gal}.
\begin{lem}\label{cond_4_equiv_lem}

Let $A$ be an unital algebra. Suppose that finite group $G$ acts on $A$.  Then following statements:
\begin{enumerate}
\item $\mathrm{can}_G: A\otimes_{A^G} A \rightarrow \mathrm{Map}(G, A)$ defined by (\ref{canonical_map_defn_eqnhp}) is bijection;
\item There are elements  $b_i, a_i \in A$ ($i=1,...,n$) such that
\begin{equation}\label{cond_4_equiv_lem_eq1}
\sum_{i =1, ..., n} a_ib_i = 1_A,
\end{equation}
\begin{equation}\label{cond_4_equiv_lem_eq2}
\sum_{i =1, ..., n} a_i(gb_i) = 0 \ \forall g\in G \ (g \ \mathrm{is \ nontrivial});
\end{equation}
\end{enumerate}
are equivalent.
\end{lem}
\begin{proof}
\begin{enumerate}
\item $=>$ Denote by $e\in G$ unity of $G$. Let $f\in \mathrm{Map}(G, A)$ be such that
\begin{equation}\nonumber
f(e)=1_A;
\end{equation}
\begin{equation}\nonumber
f(g)=0; \ (g\ne e).
\end{equation}
From bijection $A\otimes_{A^G} A \approx \mathrm{Map}(G, A)$ it follows that there are elements $a_1, ..., a_n, b_1, ..., b_n \in A$ such that $\sum_{i =1, ..., n} a_i\otimes b_i$ corresponds to $f$ i.e.
\begin{equation}\nonumber
f(g) = \sum_{i =1, ..., n} a_i(gb_i).
\end{equation}
It is clear that elements $a_1, ..., a_n, b_1, ..., b_n$ satisfy conditions (\ref{cond_4_equiv_lem_eq1}), (\ref{cond_4_equiv_lem_eq2})
\item $<=$  Let us enumerate elements of $G$, i.e $G = \{g_1, ..., g_{|G|}\}$. $a_1, ..., a_n, b_1, ..., b_n$ satisfy conditions (\ref{cond_4_equiv_lem_eq1}), (\ref{cond_4_equiv_lem_eq2}), and let be $f\in \mathrm{Map}(G, A)$ be any map from $G$ to $A$; and $x \in A\otimes_{A^G} A$ is such that
\begin{equation}\nonumber
x=\sum_{i=1,...,|G|}f(g)a_i\otimes g^{-1}b_i.
\end{equation}
From (\ref{cond_4_equiv_lem_eq1}), (\ref{cond_4_equiv_lem_eq2}) it follows that $f = \mathrm{can}_G(x)$ So $\mathrm{can}_G$ is map onto.
\end{enumerate}
\end{proof}

\begin{defn}\label{g_galois_defn} Let $G$ be a finite group. Suppose that $ H = \mathrm{Map}(G,\mathbb{C})$ and $H$ has a natural structure of Hopf algebra. Any $H$-Galois extension $A \rightarrow B$ is said to be $G$-{\it Galois
extension}.
\end{defn}
\subsection{Resume}
Above results shows that a good generalization of noncommutative covering projections requires following ingredients:
\begin{itemize}
\item Analog of  $\left(F_1 = A \otimes_{A^{\mathrm{co}H}} -, G_1 = \left(-\right)^{\mathrm{co}H}\right)$,
\item Analog of definition \ref{hoph_galois_def},
\item Constructive method for checking conditions of definitions like lemma \ref{cond_4_equiv_lem},
\item Examination of the theory for the commutative case,
\item Nontrivial noncommutative examples.
\end{itemize}

\section{Hermitian modules and functors}
In this section we consider an analogue of the $A \otimes_B - : \ _B\mathcal{M}\to _A\mathcal{M}$ functor or an algebraic generalization of continuous maps. Following text is in fact a citation of \cite{rieffel_morita}.

 \begin{defn}
\cite{rieffel_morita} Let $B$ be a $C^*$-algebra. By a (left) {\it Hermitian $B$-module} we will mean the Hilbert space $H$ of a non-degenerate *-representation $A \rightarrow B(H)$. Denote by $\mathbf{Herm}(B)$ the category of Hermitian $B$-modules.
 \end{defn}
\begin{empt}
\cite{rieffel_morita} Let $A$, $B$ be $C^*$-algebras. In this section we will study some general methods for construction of functors from  $\mathbf{Herm}(B)$ to  $\mathbf{Herm}(A)$.
\end{empt}
\begin{defn} \cite{rieffel_morita}
Let $B$ be a $C^*$-algebra. By (right) {\it pre-$B$-rigged space} we mean a vector space, $X$, over complex numbers on which $B$ acts by means of linear transformations in such a way that $X$ is a right $B$-module (in algebraic sense), and on which there is defined a $B$-valued sesquilinear form $\langle,\rangle_X$ conjugate linear in the first variable, such that
\begin{enumerate}
\item $\langle x, x \rangle_B \ge 0$
\item $\left(\langle x, y \rangle_X\right)^* = \langle y, x \rangle_X$
\item $\langle x, yb \rangle_B = \langle x, y \rangle_Xb$
\end{enumerate}
\end{defn}
\begin{empt}
It is easily seen that if we factor a pre-$B$-rigged space by subspace of the elements $x$ for which $\langle x, x \rangle_B = 0$, the quotient becomes in a natural way a pre-$B$-rigged space having the additional property that inner product is definite, i.e. $\langle x, x \rangle_X > 0$ for any non-zero $x\in X$. On a pre-$B$-rigged space with definite inner product we can define a norm $\|\|$ by setting
\begin{equation}\label{rigged_norm_eqn}
\|x\|=\|\langle x, x \rangle_X\|^{1/2}, \
\end{equation}
From now on we will always view a  pre-$B$-rigged space  with definite inner product as being equipped with this norm. The completion of $X$ with this norm is easily seen to become again a pre-$B$-rigged space.
\end{empt}
\begin{defn}
\cite{rieffel_morita} Let $B$ be a $C^*$-algebra. By a {\it $B$-rigged space} or {\it Hilbert $B$-module} we will mean a pre-$B$-rigged space, $X$, satisfying the following conditions:
\begin{enumerate}
\item If $\langle x, x \rangle_X\ = 0$ then $x = 0$, for all $x \in X$,
\item $X$ is complete for the norm defined in (\ref{rigged_norm_eqn}).
\end{enumerate}
\end{defn}
\begin{rem}
In many publications the "Hilbert $B$-module" term is used instead "rigged $B$-module".
\end{rem}
\begin{empt}
Viewing a $B$-rigged space as a generalization of an ordinary Hilbert space, we can define what we mean by bounded operators on a $B$-rigged space.
\end{empt}
\begin{exm}\label{fin_rigged_exm}
Let $A$ be a $C^*$-algebra and a finite group acts on $A$. Then $A$ is a $A^G$-rigged space on which is defined following $A^G$ valued form
 \begin{equation}\label{inv_scalar_product}
 \langle x, y \rangle_A = \frac{1}{|G|} \sum_{g \in G} g(x^*y).
 \end{equation}
 Since given by \ref{inv_scalar_product} sum is $G$-invariant we have $ \langle x, y \rangle_A \in A^G$.
\end{exm}

\begin{defn}\cite{rieffel_morita}
Let $X$ be a $B$-rigged space. By a {\it bounded operator} on $X$ we mean a linear operator, $T$, from $X$ to itself which satisfies following conditions:
\begin{enumerate}
\item for some constant $k_T$ we have
\begin{equation}\nonumber
\langle Tx, Tx \rangle_X \le k_T \langle x, x \rangle_X, \ \forall x\in X,
\end{equation}
or, equivalently $T$ is continuous with respect to the norm of $X$.
\item there is a continuous linear operator, $T^*$, on $X$ such that
\begin{equation}\nonumber
\langle Tx, y \rangle_X = \langle x, T^*y \rangle_X, \ \forall x, y\in X.
\end{equation}
\end{enumerate}
It is easily seen that any bounded operator on a $B$-rigged space will automatically commute with the action of $B$ on $X$ (because it has an adjoint). We will denote by $\mathcal{L}(X)$ (or $\mathcal{L}_B(X)$ there is a chance of confusion) the set of all bounded operators on $X$. Then it is easily verified than with the operator norm $\mathcal{L}(X)$ is a $C^*$-algebra.
\end{defn}
\begin{defn}\cite{pedersen:ca_aut} If $X$ is a $B$-rigged module then
denote by $\theta_{\xi, \zeta} \in \mathcal{L}_B(X)$   such that
\begin{equation}\nonumber
\theta_{\xi, \zeta} (\eta) = \zeta \langle\xi, \eta \rangle_X , \ (\xi, \eta, \zeta \in X)
\end{equation}
Norm closure of  a generated by such endomorphisms ideal is said to be the {\it algebra of compact operators} which we denote by $\mathcal{K}(X)$. The $\mathcal{K}(X)$ is an ideal of  $\mathcal{L}_B(X)$. Also we shall use following notation $\xi\rangle \langle \zeta \stackrel{\text{def}}{=} \theta_{\xi, \zeta}$.
\end{defn}

\begin{defn}\cite{rieffel_morita}\label{corr_defn}
Let $A$ and $B$ be $C^*$-algebras. By a {\it Hermitian $B$-rigged $A$-module} we mean a $B$-rigged space, which is a left $A$-module by means of *-homomorphism of $A$ into $\mathcal{L}_B(X)$.
\end{defn}
\begin{rem}
Hermitian $B$-rigged $A$-modules are also named  as {\it $B$-$A$-correspondences} (See, for example \cite{kakariadis:corr}).
\end{rem}
\begin{empt}\label{herm_functor_defn}
Let $X$ be a Hermitian $B$-rigged $A$-module. If $V\in \mathbf{Herm}(B)$ then we can form the algebraic tensor product $X \otimes_{B_{\mathrm{alg}}} V$, and equip it with an ordinary pre-inner-product which is defined on elementary tensors by
\begin{equation}\nonumber
\langle x \otimes v, x' \otimes v' \rangle = \langle \langle x',x \rangle_B v, v' \rangle_V.
\end{equation}
Completing the quotient $X \otimes_{B_{\mathrm{alg}}} V$ by subspace of vectors of length zero, we obtain an ordinary Hilbert space, on which $A$ acts (by $a(x \otimes v)=ax\otimes v$) to give a  *-representation of $A$. We will denote the corresponding Hermitian module by $X \otimes_{B} V$. The above construction defines a functor $X \otimes_{B} -: \mathbf{Herm}(B)\to \mathbf{Herm}(A)$ if for $V,W \in \mathbf{Herm}(B)$ and $f\in \mathrm{Hom}_B(V,W)$ we define $f\otimes X \in \mathrm{Hom}_A(V\otimes X, W\otimes X)$ on elementary tensors by $(f \otimes X)(x \otimes v)=x \otimes f(v)$.	We can define action of $B$ on $V\otimes X$ which is defined on elementary tensors by
\begin{equation}\nonumber
b(x \otimes v)= (x \otimes bv) = x b \otimes v.
\end{equation}
\end{empt}
\section{Strong and/or weak completion}
In this section we follow to \cite{pedersen:ca_aut}.
\begin{defn}\cite{pedersen:ca_aut}
Let $A$ be a $C^*$-algebra.
The {\it strict topology} on $M(A)$ is the topology generated by seminorms $\vertiii{x}_a = \|ax\| + \|xa\|$, ($a\in A$). If $x \in M(A)$  and sequence of partial sums $\sum_{i=1}^{n}a_i$ ($n = 1,2, ...$), ($a_i \in A$) tends to $x$ in strict topology then we shall write
\begin{equation}
x = \sum_{i=1}^{\infty}a_i.
\end{equation}
\end{defn}

\begin{defn}\cite{pedersen:ca_aut}
Let $B \in B(H)$ be a $C^*$-algebra. Denote by $B''$ the strong closure of $B$ in $B(H)$. $B''$ is an unital weakly closed $C^*$-algebra and if $B$ acts non-degenerately on $H$ then  $B''$ is the {\it bicommutant} of $B$. Any strongly (=weakly) closed algebra is said to be a {\it von Neumann algebra}.
\end{defn}
\begin{defn}\cite{pedersen:ca_aut}
For any $x\in B(H)$ element $|x| \stackrel{\text{def}}{=} (xx^*)^{1/2}$ is said to be the {\it absolute value of} $x$.
\end{defn}
\begin{empt}\cite{pedersen:ca_aut}
For each $x\in B(H)$ we define the {\it range projection} of $x$ (denoted by $[x]$) as projection on closure of $xH$. If $x\ge 0$ then the sequence $\left(\left((1 /n) +x\right)^{-1}x\right)$ is monotone increasing to $[x]$.  If $p$ and $q$ are projections then $p \vee q = [p + q]$ and thus $p \wedge q = 1 - \left[2 - \left(p+q\right)\right]$. Similarly we have $p \setminus q = p - p\wedge q$. Since $[x]H$ is the orthogonal complement of the null space of $x^*$ we have $[x]=[xx^*]$. If $\mathcal{M}$ is a von Neumann algebra in $B(H)$ then $[x]\in \mathcal{M}$ for any $x\in \mathcal{M}$. We next prove a {\it polar decomposition}.
\end{empt}

\begin{prop}\cite{pedersen:ca_aut}
For each element $x$ in   a von Neumann algebra $\mathcal{M}$ there is a unique partial isometry $u\in \mathcal{M}$ with $uu^*=[|x|]$ and  $x=|x|u$.
\end{prop}
\begin{proof}
Consider the sequence $u_n =x\left(\left(1/n\right)+ |x|\right)^{-1}$. Since $x=x[|x|]$ we have $u_n=u_n[x]$. A short computation shows that
\begin{equation}
(u_n - u_m)^*(u_n-u_m)=\left(\left(\left(1/n\right)+|x|\right)^{-1}-\left(\left(1/m\right)+|x|\right)^{-1}\right)|x|^2
\end{equation}
and this tends strongly, hence weakly to zero by spectral theory. It follows that $\{u_n\}$ is strongly convergent to an element $u\in \mathcal{M}$ with $u[|x|]=u$. Since $\{u_n|x|\}$ is norm convergent to $x$ we have $x=u|x|$. Then $x^*x= |x|u^*u|x|$ which implies that $u^*u \ge [|x|]$. Hence $u^*u = [|x|]$, in particular $u$ is a partial isometry. If $x = v|x|$ then from $v|x|=u|x|$ we get $v=v[|x|]=u$, so $u$ is unique.
\end{proof}

\begin{empt}\label{ortogonalization_of_algebra}
Let $I$ be a finite or countable set of indices, $B \in B(H)$ be a $C^*$- algebra, $\{e_i\} \subset B$, ($i\in I$) is such that
\begin{equation}\label{algebra_sum}
\sum_{i\in I}^{}e^*_ie_i = 1_{M(B)}
\end{equation}
with respect to the strict topology.
\newline
Let $e_i = v_i|e_i|$ be the polar decomposition, we define elements $u_i \in B''$ such that
\begin{equation}\label{orth_alg_defn}
u_i = \left([e_i] \setminus \left([e_1] \vee[e_2] \vee ... \vee [e_{i-1}]] \right) \right)v_i.
\end{equation}
From $\left([e_i]| \setminus \left([e_1] \vee[e_2] \vee ... \vee [e_{i-1}]] \right)\right)v_j = 0$,  ($j =1, ..., i-1$) it follows that $u^*_iv_j =  u^*_i \left([e_i]| \setminus \left([e_1] \vee[e_2] \vee ... \vee [e_{i-1}]] \right) \right)v_j=0$, hence
\begin{equation}\label{ji_orth}
 u^*_iu_j=0, \ \forall i \neq j
\end{equation}
and
\begin{equation}\label{ji_orth1}
[u_1] \vee ... \vee [u_i] = [u_1 + ... + u_i].
\end{equation}
From
\ref{orth_alg_defn} it follows that
\begin{equation}
[u_i] = [e_i] \setminus \left([e_1] \vee[e_2] \vee ... \vee [e_{i-1}]] \right).
\end{equation}
Hence
\begin{equation}
[u_1] \vee ... \vee [u_i] = [e_1] \vee ... \vee [e_{i}].
\end{equation}

From \ref{algebra_sum} it follows that $\bigvee_i [e_i] = 1_{B''} = 1_{M(B)}$. So $\bigvee_i [u_i] = 1_{B''}$. Because any $u_i$ is a partial isometry and from \ref{ji_orth} it follows that
\begin{equation}\label{algebra_usum}
\sum_{i\in I}^{}u^*_iu_i = 1_{B''}
\end{equation}
\end{empt}
\begin{defn}\label{vn_ort_defn}
We say that $\{u_i\}$ is the {\it von Neumann orthogonalization} of $\{e_i\}$ because $u_iH \perp u_jH$, ($i \neq j$).
It is easy to check that
\begin{equation}\label{u_conatis_e}
u_iH\subset e_iH.
\end{equation}
\end{defn}


\begin{defn}
Let $_AX_B$ be a  Hermitian $B$-rigged $A$-module, $B \to B(H)$ a faithful representation. For any $h \in H$  we define a seminorm $\vertiii{}_h$ on $_AX_B$ such that
\begin{equation*}
\vertiii{\xi}_h = \|\langle \xi, \xi \rangle_Bh\|.
\end{equation*}
Completion of $_AX_B$ with respect to above seminorms is said to be the {\it strong completion}. Denote by $X''$ or $_AX''_B$ the strong completion. There is the natural scalar product $\langle,\rangle_{X''}$ such that
\begin{equation}
\langle \xi, \zeta \rangle_{X''} \in B'',  \ \forall\xi,\zeta \in X''.
\end{equation}
\end{defn}
\begin{empt}
Since $X\otimes_B H$ is norm complete there is a following natural $B$-isomorphism
\begin{equation}\label{von_neum_iso}
X\otimes_B H \approx X''\otimes_{B''} H.
\end{equation}
\end{empt}

\section{Galois rigged modules}
\begin{defn}\label{herm_a_g_defn}

Let $A$ be a $C^*$-algebra, $G$ is a finite or countable group which acts on $A$. We say that  $H \in \mathbf{Herm}(A)$ is a {\it $A$-$G$ Hermitian module} if
\begin{enumerate}
\item Group $G$ acts on $H$ by unitary $A$-linear isomorphisms,
\item There is a subspace $H^G \subset H$ such that
\begin{equation}\label{g_act}
H = \bigoplus_{g\in G}gH^G.
\end{equation}
\end{enumerate}
Let $H$, $K$ be  $A$-$G$ Hermitian modules, a morphism $\phi: H\to K$ is said to be a $A$-$G$-morphism if $\phi(gx)=g\phi(x)$ for any $g \in G$. Denote by $\mathbf{Herm}(A)^G$ a category of  $A$-$G$ Hermitian modules and $A$-$G$-morphisms.
\end{defn}
\begin{rem}
Condition 2 in the above definition is introduced because any topological covering projection $\widetilde{\mathcal{X}} \to \mathcal{X}$ commutative $C^*$ algebras $C_0\left(\widetilde{\mathcal{X}}\right)$, $C_0\left(\mathcal{X}\right)$ satisfies it with respect to the group of covering transformations $G(\widetilde{\mathcal{X}}| \mathcal{X})$. (See \ref{comm_exm})
\end{rem}

\begin{defn}
Let $H$ be $A$-$G$ Hermitian module, $B\subset M(A)$ is sub-$C^*$-algebra such that $(ga)b = g(ab)$,  $b(ga) = g(ba)$, for any $a\in A$, $b \in B$, $g \in G$.
There is a functor $(-)^G: \mathbf{Herm}(A)^G \to\mathbf{Herm}(B)$ defined by following way
\begin{equation}
H \mapsto H^G.
\end{equation}
This functor is said to be the {\it invariant functor}.
\end{defn}

\begin{defn}
Let $_AX_B$ be a Hermitian $B$-rigged $A$-module, $G$ is finite or countable group such that
\begin{itemize}
\item $G$ acts on $A$ and $X$,
\item Action of $G$ is equivariant, i.e $g (a\xi) = (ga) (g\xi)$ , and $B$ invariant, i.e $g(\xi b)=(g\xi)b$ for any $\xi \in X$, $b \in B$, $a\in A$, $g \in G$,
\item Inner-product  of $G$ is equivariant, i.e $\langle g\xi, g \zeta\rangle_X = \langle\xi, \zeta\rangle_X$ for any $\xi, \zeta \in X$,  $g \in G$.
\end{itemize}
Then we say that  $_AX_B$ is a {\it $G$-equivariant $B$-rigged $A$-module}.
\end{defn}
\begin{empt}
Let $_AX_B$ be a  $G$-equivariant $B$-rigged $A$-module. Then for any  $H\in \mathbf{Herm}(B)$ there is an action of $G$  on $X\otimes_B H$ such that
\begin{equation}
g \left(x \otimes \xi\right) = \left(x \otimes g\xi\right).
\end{equation}

\end{empt}

\begin{defn}\label{inf_galois_defn}
Let $_AX_B$ be a $G$-equivariant $B$-rigged $A$-module. We say that  $_AX_B$ is {\it $G$-Galois $B$-rigged $A$-module} if it satisfies following conditions:
\begin{enumerate}
\item  $X \otimes_B H$ is a $A$-$G$ Hermitian module, for any $H \in \mathbf{Herm}(B)$,
\item A pair   $\left(X \otimes_B -, \left(-\right)^G\right)$ such that
\begin{equation}\nonumber
X \otimes_B -: \mathbf{Herm}(B) \to \mathbf{Herm}(A)^G,
\end{equation}
\begin{equation}\nonumber
(-)^G: \mathbf{Herm}(A)^G \to \mathbf{Herm}(B).
\end{equation}

is a pair of inverse equivalence.

\end{enumerate}

\begin{rem}
Above definition is an analog of the theorem \ref{hoph_galois_def_thm} and the definition \ref{hoph_galois_def}.
\end{rem}
\end{defn}
Following theorem is an analog of the lemma \ref{cond_4_equiv_lem} and it is very close to theorems described in \cite{miyashita_infin_outer_gal}, \cite{takeuchi:inf_out_cov}.
\begin{thm}\label{main_lem}
Let $A$ and $\widetilde{A}$ be $C^*$-algebras,  $_{\widetilde{A}}X_A$ be a $G$-equivariant $A$-rigged $\widetilde{A}$-module. Let $I$ be a finite or countable set of indices,  $\{e_i\}_{i\in I} \subset M(A)$, $\{\xi_i\}_{i\in I} \subset \ _{\widetilde{A}}X_A$ such that
\begin{enumerate}
\item
\begin{equation}\label{1_mb}
1_{M(A)} =  \sum_{i\in I}^{}e^*_ie_i,
\end{equation}
\item
\begin{equation}\label{1_mkx}
1_{M(\mathcal{K}(X))} = \sum_{g\in G}^{} \sum_{i \in I}^{}g\xi_i\rangle \langle g\xi_i ,
\end{equation}
\item
\begin{equation}\label{ee_xx}
\langle \xi_i, \xi_i \rangle_X = e_i^*e_i,
\end{equation}
\item
\begin{equation}\label{g_ort}
\langle g\xi_i, \xi_i\rangle_X=0, \ \text{for any nontrivial} \ g \in G.
\end{equation}
\end{enumerate}
Then $_{\widetilde{A}}X_A$ is a $G$-Galois $A$-rigged $\widetilde{A}$-module.

\end{thm}
\begin{proof}
For any $H\in \mathbf{Herm}(A)$ we construct $\left(X \otimes_A H\right)^G \subset X \otimes_A H$,   such that
\begin{equation}\nonumber
X \otimes_A H = \bigoplus_{g\in G}g\left(X \otimes_A H\right)^G.
\end{equation}
\begin{equation}\nonumber
H \approx \left(X \otimes_A H\right)^G.
\end{equation}

Consider following weak (=strong) limits

\begin{equation}\nonumber
v_i = \lim_{n \to \infty} \left(\left(1/n\right)+ |e_i|\right)^{-1}e_i \in A'',  \ e_i= v_i|e_i|,
\end{equation}
\begin{equation}\nonumber
\xi'_{i} =  \lim_{n \to \infty} \left(\left(1/n\right)+ |e_i|\right)^{-1}\xi_i \in X''.
\end{equation}
From \eqref{ee_xx} it follows that
\begin{equation}\label{xi_prod}
\langle \xi'_i, \xi'_i \rangle_{X''} = v^*_iv_i.
\end{equation}
Let $\{u_i\}_{i \in I} \subset B''$ be the von Neumann orthogonalization of $\{e_i\}_{i \in I}$. Let $\{\zeta''_i\}_{i\in I}\subset X''$ be such that

\begin{equation}\nonumber
\xi''_i = \xi'_iu_i.
\end{equation}
Then from definition \ref{vn_ort_defn} and \eqref{xi_prod} it follows that
\begin{equation}\label{orth_u}
\langle \xi''_i, \xi''_j \rangle_{X''} = u^*_iu_j.
\end{equation}
From (\ref{ji_orth}), (\ref{orth_u}) it follows that
\begin{equation}\label{ort_i}
\langle \xi''_i, \xi''_j \rangle_{X''} = u^*_iu_j = 0, \ (i \neq j).
\end{equation}
From \eqref{1_mkx} and \eqref{g_ort} if follows that
\begin{equation}\label{1_mkx''}
1_{M(\mathcal{K}(X''))} = \sum_{g\in G}^{} \sum_{i=1}^{\infty}g\xi''_i\rangle \langle g\xi''_i ,
\end{equation}
\begin{equation}\label{g_ort_p}
\langle g\xi''_i, \xi''_i \rangle_{X''} = 0, \ \text{for any nontrivial} \ g \in G,
\end{equation}
From \eqref{von_neum_iso} it follows that we can define $\xi''_i \otimes h \in X \otimes_B H$.
Define  $\left(X \otimes_A H\right)^G$ as a norm closure of a generated by elements $\xi''_i \otimes h$ $A$-module, ($i \in I, h \in H$).
It is clear that $g\left(X \otimes_A H\right)^G$ is generated by elements $g\xi''_i \otimes h$, ($i \in I, h \in H$).
From \eqref{ort_i} \eqref{1_mkx''}, \eqref{g_ort_p} it follows that
\begin{equation}\nonumber
X \otimes_B H = \bigoplus_{g \in G} g \left(X \otimes_B H\right)^G
\end{equation}
So $X$ satisfies condition 1 of definition \ref{inf_galois_defn}.
Let $H\in \mathbf{Herm}(\widetilde{A})^G$ be a Hermitian $\widetilde{A}$-$G$-module. For any $h_1, h_2 \in H^G$ we have
\begin{equation}\label{ort_2}
\langle \xi''_i \otimes h_1, \xi''
_j \otimes h_2 \rangle =  \langle \langle \xi''_i,\xi''_j \rangle_{X''} h_1,  h_2\rangle. = \langle u^*_i u_j h_1, h_2\rangle
\end{equation}
Form (\ref{ort_i}), (\ref{ort_2}) it follows that
\begin{equation}\nonumber
\langle \zeta''_i \otimes h_1, \zeta''_j \otimes h_2 \rangle = 0, \ (i \neq j, \ h_1, h_2 \in H^G)
\end{equation}
\begin{equation}\label{g_ort''}
\langle \zeta''_i \otimes h_1, g \zeta''_i \otimes h_2 \rangle = 0, \ \text{for any nontrivial} \ g \in G.
\end{equation}
From (\ref{1_mb}) \eqref{1_mkx} it follows that

\begin{equation}\label{1_b''}
1_{B''} = \sum_{i \in I}u^*_iu_i ,
\end{equation}
\begin{equation}\nonumber
H^G = \bigoplus_{i \in I}u_iH^G.
\end{equation}
\begin{equation}\label{rep_phi}
H = \bigoplus_{g\in G} g \bigoplus_{i \in I}u_iH^G.
\end{equation}
Representation \eqref{rep_phi} supplies following natural isomorphism
\begin{equation*}
 \varphi: \left(X \otimes_A -\right) \circ \left(-\right)^G \approx 1_{\mathbf{Herm}(\widetilde{A})^G} \ ,
\end{equation*}
\begin{equation*}
\varphi\left(\sum_{g \in G} \sum_{i \in I} g\left(u_ih_{gi}\right)\right) \stackrel{\text{def}}{=} \sum_{g \in G} \sum_{i \in I} \left(g\xi''_i \otimes h_{gi}\right).
\end{equation*}
There is a natural $\theta$ isomorphism such that
\begin{equation*}
 \theta: \left(-\right)^G \circ \left(X \otimes_A -\right) \ \approx 1_{\mathbf{Herm}(A)} \ ,
\end{equation*}

\begin{equation*}
\theta\left( \sum_{i \in I} \xi''_i \otimes h_i\right) \stackrel{\text{def}}{=}  \sum_{i \in I} u_i h_i.
\end{equation*}
So $\left(X \otimes_B -, \left(-\right)^G\right)$ is a pair of inverse equivalence.
\end{proof}
\begin{defn}
Consider a situation from the theorem \ref{main_lem}. Let us consider two specific cases
\begin{enumerate}
\item $e_i \in A$ for any $i \in I$,
\item $\exists i \in I \ e_i \notin A$.
\end{enumerate}

Norm completion of the generated by operators
\begin{equation*}
g\xi_i^* \rangle \langle g \xi_i \ a; \ g \in G, \ i \in I, \ \begin{cases}
   a \in M(A), & \text{in case 1},\\
    a \in A, & \text{in case 2}.
  \end{cases}
\end{equation*}
algebra is said to be the {\it subordinated to $\{\xi_i\}_{i \in I}$ algebra}. If $\widetilde{A}$ is the subordinated to $\{\xi_i\}_{i \in I}$ then
\begin{enumerate}
\item $G$ acts on $\widetilde{A}$ by following way
\begin{equation*}
g \left( \ g'\xi_i \rangle \langle g' \xi_i \ a \right) =  gg'\xi_i \rangle \langle gg' \xi_i \ a; \ a \in M(A).
\end{equation*}
\item $X$ is a left $A$ module, moreover $_{\widetilde{A}}X_A$ is a  $G$-Galois $A$-rigged $\widetilde{A}$-module.
\item There is a natural $G$-equivariant *-homomorphism $\varphi: A \to M\left(\widetilde{A}\right)$, $\varphi$ is equivariant, i.e.
\begin{equation}
 \varphi(a)(g\widetilde{a})= g \varphi(a)(\widetilde{a}); \ a \in A, \ \widetilde{a}\in \widetilde{A}.
\end{equation}
\end{enumerate}
A quadruple $\left(A, \widetilde{A}, _{\widetilde{A}}X_A, G\right)$ is said to be a {\it Galois quadruple}. The group $G$ is said to be a {\it group Galois transformations} which shall be denoted by $G\left(\widetilde{A}\ | \ A\right)=G$.
\end{defn}
\begin{rem}
Henceforth subordinated algebras only are regarded as noncommutative generalizations of covering projections.
\end{rem}
\begin{defn}
If $G$ is finite then bimodule $_{\widetilde{A}}X_A$ can be replaced with $_{\widetilde{A}}\widetilde{A}_A$ where product $\langle \ , \ \rangle_{\widetilde{A}}$ is given by \eqref{inv_scalar_product}. In this case a Galois quadruple $\left(A, \widetilde{A}, _{\widetilde{A}}X_A, G\right)=\left(A, \widetilde{A}, _{\widetilde{A}}A_A, G\right)$ can be replaced with a {\it Galois triple}  $\left(A, \widetilde{A}, G\right)$.
\end{defn}
 \begin{lem}\label{ideal_lem} Let us consider situation from theorem \ref{main_lem}. Then for any $i \in I$ there is a natural isomorphism of right $A$ modules given by
 \begin{equation}\label{ideal_ei}
  e^*_ie_iM(A) \approx \xi_i \rangle \langle \xi_i \ M(A) 
 \end{equation}
 Let $A \to B(H)$ be a Hermitian representation. There is a natural isomorphism of hermitian $A$-modules
 \begin{equation}\label{hilbert_ei}
 e^*_ie_i H \approx \xi_i \rangle \langle \xi_i \ X \otimes_A H.
 \end{equation}
 \end{lem}
 \begin{proof}
 Let $a \in M(A)$ be such that $e^*_ie_ia=0$. Then from  $(e_ia)^*(e_ia) = a^*e^*_ie_ia=0$ it follows that $e_ia=0$. Similarly from $\xi_i \rangle \langle \xi_i a = 0$ it follows that $\xi_i a \rangle \langle \xi_i a = 0$, and therefore $\xi_i a = 0$. Tnere is an equivalence $\xi_i a = 0 \Leftrightarrow\langle \xi_i a, \xi_i a \rangle=0$. But  $\langle \xi_i a, \xi_i a \rangle=a^*e^*_ie_ia$, and we have $e_ia=0$. There are following equivalences
 \begin{equation*}
 e^*_ie_ia = 0 \Leftrightarrow e_ia = 0 \Leftrightarrow \xi_i a = 0 \Leftrightarrow \xi_i \rangle \langle \xi_i a = 0; \ a \in A.
 \end{equation*}
 From above equivalences it follows \eqref{ideal_ei}. Similarly $e^*_ie_ih = 0 \Leftrightarrow e_ih = 0$. If $\zeta \otimes h \in X \otimes H$ then
 \begin{equation*}
\xi_i \rangle \langle \xi_i \ \zeta \otimes h = \xi_i \otimes  \langle \xi_i, \zeta \rangle h = \xi_i \otimes h'; \ \text{where} \ h' = \langle \xi_i, \zeta \rangle h.
 \end{equation*}
 If $\xi_i \otimes h' = 0$ then $\left(\xi_i \otimes h', \xi_i \otimes h'\right)=\left(h', \langle \xi_i, \xi_i\rangle h'\right) = \left(h', e^*_i e_i h'\right) = 0$, and therefore $e^*_ie_ih' = e_ih' = 0$.
 \end{proof}

\section{Infinite noncommutative covering projections}
\paragraph{} In case of commutative $C^*$-algebras definition \ref{inf_galois_defn} supplies algebraic formulation of infinite covering projections of topological spaces. However I think that above definition is not a quite good analogue of noncommutative covering projections. Noncommutative algebras contains inner automorphisms. Inner automorphisms are rather gauge transformations \cite{gross_gauge} than geometrical ones. So I think that inner automorphisms should be excluded. Importance of outer automorphisms was noted by  Miyashita \cite{miyashita_fin_outer_gal,miyashita_infin_outer_gal}. Example 3.9 from \cite{ivankov:nc_cov_k_hom} also proves that inner automorphisms should be excluded. It is reasonably take to account outer automorphisms only. I have set more strong condition.
\begin{defn}\label{gen_in_def}\cite{rieffel_finite_g}
Let  $A$ be $C^*$-algebra. A *-automorphism $\alpha$ is said to be {\it generalized inner} if it is given by conjugating with unitaries from multiplier algebra $M(A)$.
\end{defn}
\begin{defn}\label{part_in_def}\cite{rieffel_finite_g}
Let  $A$ be $C^*$ - algebra. A *- automorphism $\alpha$ is said to be {\it partly inner} if its restriction to some non-zero $\alpha$-invariant two-sided ideal is generalized inner. We call automorphism {\it purely outer} if it is not partly inner.
\end{defn}
Instead definitions \ref{gen_in_def}, \ref{part_in_def} following definitions are being used.
\begin{defn}
Let $\alpha \in \mathrm{Aut}(A)$ be an automorphism. A representation $\rho : A\rightarrow B(H)$ is said to be {\it $\alpha$ - invariant} if a representation $\rho_{\alpha}$ given by
\begin{equation*}
\rho_{\alpha}(a)= \rho(\alpha(a))
\end{equation*}
is unitary equivalent to $\rho$.
\end{defn}
\begin{defn}
Automorphism $\alpha \in \mathrm{Aut}(A)$ is said to be {\it strictly outer} if for any $\alpha$- invariant representation $\rho: A \rightarrow B(H) $, automorphism $\rho_{\alpha}$ is not a generalized inner automorphism.
\end{defn}
\begin{defn}\label{nc_fin_cov_pr_defn}
A Galois quadruple  $\left(A, \widetilde{A}, _{\widetilde{A}}X_A, G\right)$  is said to be a {\it noncommutative infinite
covering projection} if action of $G$ on $\widetilde{A}$ is strictly outer. Any {\it finite covering projection} is a particular case of infinite one where $G$ is finite.
\end{defn}
\begin{exm}\label{unconn_exm} {\it Boring example}. Let $A$ be a separable $C^*$-algebra, and let $G$ be an arbitrary finite or countable discrete group. Let $X \subset \mathrm{Map}(G, A)$ such that 
\begin{equation*}
	X = \left\{ f \in \mathrm{Map}(G, A) \ | \ \sum_{g\in G} \|f(g)\|^2 < \infty\right\}.
\end{equation*}
 Also $X$ is an $A$-rigged space such that an $A$-valued form is given by
\begin{equation*}
\langle f, h \rangle_X = \sum_{g \in G} f^*(g) h(g).
\end{equation*}
Let $\widetilde{A}\subset \mathrm{Map}(G, A)$ be such that for any $f \in \widetilde{A}$ and $\varepsilon > 0$ there are only finitely many elements $g \in G$ such that $\|f(g)\| > \varepsilon$. There is a natural structure of $C^*$-algebra on $\widetilde{A}$ given by
\begin{equation*}
(fh)(g)=f(g)h(g); \ \forall g \in G, \ \forall f, h \in \widetilde{A},
\end{equation*}
\begin{equation*}
f^*(g)=\left(f(g)\right)^*; \ \forall g \in G, \ \forall f \in \widetilde{A}.
\end{equation*}
 There is the natural action of $G$ on both $X$ and $\widetilde{A}$.
It is easy to show that  $\left(A, \widetilde{A}, _{\widetilde{A}}X_A, G\right)$  is a noncommutative infinite
covering projection. Since $G$ is arbitrary this example is boring.
\end{exm}
\begin{rem}
If $A=C_0(\mathcal{X})$ then example \ref{unconn_exm} corresponds to the natural covering projection
\begin{equation*}
\bigsqcup_{g\in G} \mathcal{X} \to \mathcal{X}.
\end{equation*}
\end{rem}
\begin{defn}
A ring is said to be {\it irreducible } if it is not a direct sum of more than one ring. A Galois quadruple  $\left(A, \widetilde{A}, _{\widetilde{A}}X_A, G\right)$ is said to be {\it irreducible} if both $A$ and $\widetilde{A}$ are irreducible.
\end{defn}

\section{Examples of infinite covering projections}
\subsection{Infinite covering projections of  locally compact topological spaces}\label{comm_exm}

\paragraph{} Let $\widetilde{\mathcal{X}}$ and $\mathcal{X}$ be locally compact topological spaces, $p : \widetilde{\mathcal{X}} \to \mathcal{X}$ is a regular covering projection such that the group of covering transformations $G = G\left(\widetilde{\mathcal{X}}|\mathcal{X}\right)$ is finite or countable. Let $I$ be a finite or countable set of indices such that there is a locally finite \cite{munkres:topology}  covering  $\mathcal{U}_i \subset \mathcal{X}$ ($i \in I$) of $\mathcal{X}$ ($\mathcal{X} = \bigcup \mathcal{U}_i$) by connected open subsets such that $p^{-1}(\mathcal{U}_i)$ ($\forall i\in I$) is a disjoint union of naturally homeomorphic to $\mathcal{U}_i$ sets. Let $1_{M(C_0(\mathcal{X}))} = \sum_{i\in I}^{}a_i$ be a partition of unity dominated by $\{\mathcal{U}_i\}$ (See \cite{munkres:topology}). The family $e_i = \sqrt{a_i}$  satisfies condition (\ref{1_mb}), i.e.
\begin{equation*}\label{e_comm}
1_{M(C_0(\mathcal{X}))} = \sum_{i \in I}^{} e^*_ie_i.
\end{equation*}
Select a connected component $\mathcal{V}_i \subset \pi^{-1}(\mathcal{U}_i)$ for any $i\in I$ and $\mathcal{V}_i \bigcap g \mathcal{V}_i = \emptyset$. Let $\xi_i \in C_0(\widetilde{\mathcal{X}})$ is such that
\begin{equation}\label{comm_1}
\xi_i(x)=\left\{
\begin{array}{c l}
    e_i\left(p\left(x\right)\right) & x \in \mathcal{V}_i \\
   0 & x \notin \mathcal{V}_i
\end{array}\right.
\end{equation}
It is easy to check that
\begin{equation}\label{comm_2}
1_{M(C_0(\widetilde{\mathcal{X}}))} = \sum_{g\in G}^{} \sum_{i\in I}^{}g\xi^*_i g\xi_i.
\end{equation}
From  $\mathcal{V}_i \bigcap g \mathcal{V}_i = \emptyset$ it follows that
\begin{equation}\label{comm_ort}
\langle g\xi_i, \xi_i\rangle_X=0, \  \text{for any nontrivial} \ g \in G.
\end{equation}

For any $\xi_i$ ($i\in I$) and any $\eta \in C_0(\widetilde{\mathcal{X}})$ there is an unique $b \in C_0(\mathcal{X})$ such that $\xi_i\eta = \xi_i b$. Denote by $\langle \xi_i, \eta \rangle \stackrel{\text{def}}{=} e^*_ib\in C_0(\mathcal{X})$, $\langle \eta, \xi_i \rangle \stackrel{\text{def}}{=} \langle \xi_i, \eta \rangle^*$ Let us define a subspace  $X \subset C_0(\widetilde{\mathcal{X}})$ such that for any $\zeta \in X$ the series
\begin{equation}\nonumber
\sum_{g\in G}\sum_{i\in I}^{}\langle \zeta, g\xi_i \rangle \langle g\xi_i, \zeta \rangle
\end{equation}
is norm convergent. Define scalar product $\langle \xi, \zeta\rangle_X$ on $X$ such that
\begin{equation}\nonumber
\langle \xi, \zeta\rangle_X = \sum_{g\in G}\sum_{i\in I}^{}\langle \xi, g\xi_i \rangle \langle g\xi_i, \zeta \rangle
\end{equation}
Natural action $G$  on $C_0(\widetilde{\mathcal{X}})$, induces action of $G$ on $X$, so $X$ is
 From (\ref{comm_2}),(\ref{comm_ort}) it follows that  $_{C_0(\widetilde{\mathcal{X}})}X_{C_0(\mathcal{X})}$  is a $G$ - Galois $C_0(\mathcal{X})$-rigged $C_0(\widetilde{\mathcal{X}})$-module. $C_0(\mathcal{X})$ is the subordinated to $\{\xi_i\}_{i \in I}$ algebra.
Since $C_0(\widetilde{\mathcal{X}})$ is commutative any *-automorphism of $C_0(\widetilde{\mathcal{X}})$ is strictly outer. So any topological infinite covering projection corresponds to an algebraic one.

\subsection{Infinite covering projection of continuous trace $C^*$-algebras}\label{cont_tr_exm}

\begin{empt}\label{irreducible_corr}{\it Irreducible representations of noncommutative covering projections}. Let $\left(A, \widetilde{A}, _{\widetilde{A}}X_A, G\right)$ be an infinite noncommutative covering projection. Let $\rho:  \widetilde{A} \to \mathcal{B}(H)$ be an irreducible representation.
Let $g \in G$ and $\rho_g:  A \to\mathcal{B}(H)$ be such  that
\begin{equation}\nonumber
\rho_g (a)= \rho(ga).
\end{equation}
So there is the action of $G$ on $\widehat{\widetilde{A}}$ such that
\begin{equation}\label{action_of_g_on_spectrum}
g \mapsto (\rho \mapsto \rho_g); \ \forall g \in G, \forall \rho \in \widehat{\widetilde{A}}.
\end{equation}

Let $I$ be a set such that $\mathrm{card} \ I=\mathrm{card} \ G$, and let $G=\bigsqcup_{i\in I}\{g_i\}$, ($g_i \in G$). Let $H^I = \oplus_{i \in I}H$ be a Hilbert direct sum.  Let us define action of $\sigma: G \times I \to I$ such that $\sigma(g, i) = j \Leftrightarrow  g_j = gg_i$.
Let $\rho_{\oplus} = \oplus_{g \in G} \rho_g: \widetilde{A}\to B (H^I)$ be such that
\begin{equation}\label{a_act_irr_sum}
\rho_{\oplus} (a)\left(h_{i_1}, ..., h_{i_n}, ...\right)= \left(\rho_{g_{i_1}}(a)h_{i_1}, ...,  (\rho_{g_{i_n}}(a)h_{i_n}, ...\right); \ (i_1,...,i_n,...\in I).
\end{equation}
Let us define such linear action of $G$ on $H^I$ that
\begin{equation}\label{g_act_irr_sum}
g \left(h_{i_1}, ..., h_{i_n}, ...\right)  = \left(h_{\sigma(g^{-1}, i_1)}, ...,h_{\sigma(g^{-1}, i_n)}, ...\right).
\end{equation}
Let denote $ah = \rho_{\oplus}(a)h$; $\forall h \in H^I, \forall a \in \widetilde{A}$. From (\ref{a_act_irr_sum}), (\ref{g_act_irr_sum}) it follows that
\begin{equation}\nonumber
g(ah)=(ga)(gh); \ \forall a\in A, \ \forall g\in G, \ \forall h\in H^I,
\end{equation}
i.e. $H^I \in \mathbf{Herm}(\widetilde{A})^G$. The space $_{\widetilde{A}}X_A$ is $G$-Galois $A$-rigged $\widetilde{A}$-module, therefore equivariant representation $\rho_{\oplus}$  defines an unique representation $\eta: A \to B(K)$, where $K = \left(H^I\right)^G \approx H$. If $\eta$ is not an irreducible then there is a nontrivial Hermitian
$A$-submodule $N  \varsubsetneq K$. From $ \mathbf{Herm}(\widetilde{A})^G  \approx  \mathbf{Herm}(A)$ it follows that $X \otimes_{A} N \varsubsetneq H^I$ is a nontrivial $\widetilde{A}$ - submodule. If we identify $H$ with first summand of $H^I=\oplus_{i \in I}H$ then $(\widetilde{A} \otimes_{A} K) \cap H \varsubsetneq  H$ is a nontrivial $A$ - submodule. This fact contradicts with that $\rho$ is irreducible. So $\eta$ is
an irreducible representation. In result we have the natural map
\begin{equation}\label{spectrum_galois_map}
\hat f : \widehat{\widetilde{A}} \to \hat A, \ (\rho \mapsto \eta)
\end{equation}
and from $\hat f(\rho) = \hat f(\rho_g)$; $ \ \forall g \in G$ it follows that
\begin{equation}\label{spectrum_galois_quot}
\hat{A} \approx \widehat{\widetilde{A}} / G.
\end{equation}

\end{empt}

\begin{empt} Let $A$ be a $C^*$-algebra. For each $x\in A_+$ the (canonical) trace of $\pi(x)$ depends only on the equivalence class of an irreducible representation $\pi: A\to B(H)$, so that we may define a function $\hat x : \hat A \to [0,\infty]$ by $\hat x(t)=\mathrm{Tr}(\pi(x))$ whenever $\left(\pi: A\to B(H)\right)\in t$. From proposition 4.4.9 \cite{pedersen:ca_aut} it follows that $\hat x$ is lower semicontinuous function in Jacobson topology.
\end{empt}
\begin{defn}\label{continuous_trace_c_a_defn}\cite{pedersen:ca_aut} We say that element $x\in A$ has {\it continuous trace} if $\hat x \in C_b(\hat A)$. We say that $C^*$-algebra has {\it continuous trace} if a set of elements with continuous trace is dense in $A$.
\end{defn}

\begin{defn}\label{abelian_element_defn}\cite{pedersen:ca_aut}
A positive element in $C^*$ - algebra $A$ is {\it abelian} if subalgebra $xAx \subset A$ is commutative.
\end{defn}
\begin{defn}\cite{pedersen:ca_aut}
We say that a $C^*$-algebra $A$ is of type $I$ if each non-zero quotient of $A$ contains non-zero
abelian element. If $A$ is even generated (as $C^*$-algebra) by its abelian elements we say
that it is of type $I_0$.
\end{defn}

\begin{prop}\label{abelian_element_proposition}\cite{pedersen:ca_aut}
A positive element $x$ in $C^*$-algebra $A$ is abelian if $\mathrm{dim}\pi(x) \le 1$ for every irreducible representation $\pi:A \to B(H)$.
\end{prop}

\begin{thm}\label{peder_id_thm} (Theorem 5.6 \cite{pedersen:ca_aut}) For each $C^*$ - algebra $A$ there is a dense hereditary ideal $K(A)$,
which is minimal among dense ideals.

\end{thm}
\begin{defn}
The ideal $K(A)$ from the theorem \ref{peder_id_thm} is said to be the {\it Pedersen ideal of $A$}. Henceforth Pedersen ideal shall be denoted by $K(A)$.
\end{defn}
\begin{prop}\label{continuous_trace_c_a_proposition}\cite{pedersen:ca_aut}
Let $A$ be a $C^*$ - algebra with continuous trace. Then
\begin{enumerate}
\item $A$ is of type $I_0$;
\item $\hat A$ is a locally compact Hausdorff space;
\item For each $t \in \hat A$ there is an abelian element $x \in A$ such that $\hat x \in K(\hat A)$ and $\hat x(t) = 1$.
\end{enumerate}
The last condition is sufficient for $A$ to have continuous trace.
\end{prop}

\begin{rem}\label{ctr_is_ccr}
From \cite{dixmier_tr}, Proposition 10, II.9 it follows that a continuous trace
$C^*$-algebra $A$ is always a $CCR$-algebra, i.e. for every irreducible
representation $\pi:A \to B(H)$ following condition hold
\begin{equation}\label{ccr_compact}
\pi(A)\approx \mathcal{K}
\end{equation}

\end{rem}

\begin{defn}\label{hull_top_defn}
Let $A$ be a $C^*$-algebra, and $\mathrm{Prim}(A)$ is a set of primitive ideals. For any subset $F \in A$ there exist subset $F^-$ such that

\begin{equation}\nonumber
F^- = \{t \in \mathrm{Prim}(A) : F \in t\}.
\end{equation}
$\mathrm{Prim}(A)$ is topological space such that for any closed subset $X \in \mathrm{Prim}(A)$, $\exists F \subset A, \ X = F^-$.

\end{defn}

\begin{prop}\label{prop_act_grp}\cite{bourbaki_sp:gt}
If a topological group $G$ acts properly on a topological space then orbit space $X/G$ is Hausdorff. If also $G$ is Hausdorff, then $X$ is Hausdorff.
\end{prop}

\begin{lem}\label{finite_ctra_cov_pr_free ation}
 $\left(A, \widetilde{A}, _{\widetilde{A}}X_A, G\right)$ be an infinite noncommutative covering projection. Suppose that $A$ is separable and has continuous trace. Then $\widetilde{A}$ is also separable and has continuous trace. There is the natural (topological) covering projection $\hat f : \widehat{\widetilde{A}} \to \hat A$ of topological spaces, $G$ acts freely on $\hat A$ and there is the natural  homeomorphism $\hat A \approx \widehat{\widetilde{A}} /G$.
\end{lem}
\begin{proof}
Suppose that $G$ does not act freely on $\widehat{\widetilde{ A}}$.  Then there are $x \in \widehat{\widetilde{ A}}$ and $g \in G$ such that $gx = x$. By definition \ref{nc_fin_cov_pr_defn} $g$ should be strictly outer.  Let $\rho: \widetilde{A} \to H$ be a representative of $x$. Then $\rho_g: \widetilde{A} \to H$ is also representative of $x$. So $\rho$ is unitary equivalent to $\rho_g$, i. e. there is unitary $U \in U(H)$ such that $\rho(a) = U\rho_g(a)U^*$ ($\forall a \in A)$. From \eqref{ccr_compact} it follows that $\rho(A) = \mathcal{K}$, $\rho(M(A)) = B(H)$, $\rho(U(M(A))) = U(H)$. So there is $u\in U(M(A))$ such that $\rho(u)=U$, and we have $\rho_g(a) = \rho(u)\rho(a)\rho(u^*)$. It means that $g$ is inner with respect to $\rho: \widetilde{A} \to H$, so action of $g$ is not strictly outer.  This contradiction proves that $G$ acts freely on $\widehat{\widetilde{A}}$, and $\widehat{\widetilde{A}}$ is a Hausdorff space. Let $\varphi: \widehat{\widetilde{A}} \to \hat A$ be a natural (topological) covering projection, $t\in \widehat{\widetilde{A}}$, $s = \varphi(t)\in \hat A$. Let $\mathcal{U} \subset \widehat{\widetilde{A}}$ be such that $t \in \mathcal{U}$ and $\mathcal{U} \bigcap g\mathcal{U} = \emptyset$ for any nontrivial $g\in G$. Let $f \in K(C_0(\widehat{\widetilde{A}}))_+$ be such that  $f(t)=1$ and $\mathrm{supp}(f)\subset\mathcal{U}$. From proposition \ref{continuous_trace_c_a_proposition} it follows that there is an abelian element $x \in A$ such that $\hat x \in K(\hat A)$ and $\hat x(s) = 1$. Since $x$ is abelian $\dim \ \pi(x) \le 1$ for any irreducible $\pi: A \to B(H)$. Let $y = fx \in K\left(\widetilde{A}\right)$. From $\mathcal{U} \bigcap g\mathcal{U} = \emptyset$ it follows that
\begin{equation}\label{irr_y}
\rho(y) \neq 0 \Rightarrow \rho_g(y)=0
\end{equation}
for any irreducible $\widetilde{A} \to B(H)$ and any nontrivial $g\in G$. From \eqref{irr_y}, $\dim \ \pi(x) \le 1$  and construction \ref{irreducible_corr} it follows that  $\dim \ \rho(y) \le 1$ for any irreducible $\rho: \widehat{\widetilde{A}} \to B(H)$. From $f(t)=1$, and $\hat x(s)=1$ it follows that $\hat y(t)=1$. So $\widetilde{A}$ satisfies condition 3 of proposition \ref{continuous_trace_c_a_proposition}, and therefore $\widetilde{A}$ has continuous trace.
\end{proof}

\begin{lem}\label{finite_ctra_cov_pr}
  Let $A$ be a separable $C^*$-algebra and $A$ has continuous trace. Let $\mathcal{X} \to \hat A$ be a (topological) covering projection such that $G = G\left(\mathcal{X} \ | \ \hat A\right)$ is a finite or countable group. Then there is a noncommutative covering projection $\left(A, \widetilde{A}, _{\widetilde{A}}X_A, G\right)$ such that there are a natural homeomorphism $\widehat{\widetilde{A}} \approx \mathcal{X}$ and the algebraic tensor product $C_0(\mathcal{X}) \otimes_{C_0(A)} A$ is a dense subalgebra of $\widetilde{A}$.
\end{lem}
\begin{proof}
Let $X_1 = C_0(\mathcal{X}) \otimes_{C_0(\hat A)} A$ be the algebraic tensor product. The space $X_0$ has a natural structure of involutive algebra such that for any $x\in C_0(\mathcal{X})$ the  element $x \otimes 1_{M(A)}$ commutes with any element in $X_0$. There is the natural involutive homomorphism $A \to M(X_1)$. Action of $G$ on $C_0(\mathcal{X})$ induces the natural action of $G$ on $X_0$. Let $X_1 \subset X_0$ be such that
\begin{equation*}
X_1 = \left\{x \in  X_0 \ | \ \exists a\in A; \ \sum_{g \in G}g\left(x^*x\right) =a  \right\}.
\end{equation*}
where a sum of the series means the strict convergence. Group $G$ naturally acts on $X_1$ and $X_1$ is a pre-$A$-rigged space such that the $A$-valued form is given by
\begin{equation*}
\langle x, y \rangle_{A} = \sum_{g \in G} g(x^*y)\in A
\end{equation*}
where a sum of the series means the strict convergence. Let $X_A$ be the norm completion of $X_1$, so $X_A$ is an $A$-rigged space. Similarly to \ref{comm_exm} we define a finite or countable set $I_0$, subsets $\{\mathcal{U}_{i_0}\}_{i_0\in I_0}\in \hat A$, $\{\mathcal{V}_{i_0}\}_{i_0\in I_0}\in  \mathcal{X}$, $e'_{i_0} \in C_0(\hat A)$, $\xi'_{i_0} \in C_0(\mathcal{X})$, such that $\mathrm{supp} \ e_{i_0} \subset \mathcal{U}_{i_0}$, $\mathrm{supp} \ \xi_{i_0} \subset \mathcal{V}_{i_0}$ and
\begin{equation*}
1_{M(C_0(\hat A))} = \sum_{i \in I_0} e'^*_{i_0}e'_{i_0},
\end{equation*}
\begin{equation*}
1_{M(C_0(\mathcal{X}))} = \sum_{g\in G}^{} \sum_{i_0\in I_0}^{}g\xi'^*_{i_0} g\xi'_{i_0}.
\end{equation*}
\begin{equation*}
\langle g\xi_{i_0}, \xi'_{i_0}\rangle_{C_0(\hat A)}=0, \  \text{for any nontrivial} \ g \in G.
\end{equation*}
Suppose that  $\mathcal{U}_{i_0}$ is contractible for any $i_0\in I_0$.  From general topology it follows that any vector bundle over $\mathcal{U}_{i_0}$ is trivial. From triviality of any vector bundle it follows that $C_0(\mathcal{U}_{i_0}) \otimes_{C_0(\hat A)}A \approx C_0(\mathcal{U}_{i_0}) \otimes \mathcal{K}$  for any $i_0 \in I_0$. Let $\mathcal{K} = \mathcal{K}(H_0)$ where $H_0$ has a basis  $\{x_j\}_{j\in J}$, where $J$ is a finite or countable set.  For any $j\in J$ denote by $p_j \in B(H_0)$ a one dimensional projector on $x_j$. Let $I = I_0 \times J$, $e_i = e'_{i_0} \otimes p_j $, $\xi_i = \xi'_{i_0} \otimes p_j$ where $i = (i_0, j)$. Direct calculatios shows that $e_i$, $\xi_i$ satisfy to the theorem \ref{main_lem}. Let $\widetilde{A}$ be a $C^*$-algebra subordinated to $e_i$, $\xi_i$. From the construction it follows that there is a homeomorphism $\widehat{\widetilde{A}} \approx \mathcal{X}$ and algebraic tensor product $C_0(\mathcal{X}) \otimes_{C_0(A)} A \subset \widetilde{A}$ is a dense in $\widetilde{A}$.
\end{proof}
\begin{rem}
Any commutative $C^*$-algebra has continuous trace. So described is subsection \ref{comm_exm} case is a particular case of \ref{cont_tr_exm}.
\end{rem}

\subsection{Infinite covering projection of noncommutative torus}\label{nc_torus_covering}

\begin{empt}\label{inf_cov_nct} {\it Construction of covering $C^*$-algebra}.

A noncommutative torus \cite{varilly:noncom} $A_{\theta}$ is a $C^*$-algebra generated by two unitary elements ($u, v \in U(A_{\theta})$) such that
\begin{equation}\nonumber
uv = e^{2\pi i\theta}vu, \ (\theta \in \mathbb{R}).
\end{equation}
We shall construct a $C^*$-algebra $\widetilde{A}_{\theta}$ which is said to be a universal covering of noncommutative torus. This algebra is not unique but it is a representative of the unique strong Morita equivalence class. Let $u \in U(B(H))$ be an unitary operator, such that $\mathrm{sp}(u)= \mathbb{C}^* = \{z \in \mathbb{C} \ | \ |z|=1\}$.
Let $\{\phi_i\}_{i\in \mathbb{N}}$ be an arbitrary sequence of
 (non-unique) Borel-measurable functions such that
\begin{equation}\label{root_sequence}
(\phi_i(z))^2 = z, \ (\forall z \in \mathrm{sp}(u), \ i \in \mathbb{N} ).
\end{equation}
According to the spectral theorem there is an operator $u_2 = \phi_1(u)$ such that $u^2_2=u$. Similarly we can construct $u_4 = \phi_2(u_2), ..., u_{2^n} = \phi_{n-1}(u_{2^{n-1}}), ...$ such that $u^2_{2^{n+1}}=u_{2^n}$ for any $n \in \mathbb{N}$. We have a following sequence of $C^*$-algebras
\begin{equation}
C(u) \subset C(u_2) \subset ... \subset C(u_{2^n}) \subset ... \ .
\end{equation}
For any $f \in C(\mathbb{C}^*)$ there is $f(u_{2^n})\in C(u_{2^n})$. Let $C'(u_{2^n}) \subset C(u_{2^n})$ be an ideal generated by $\{f(u_{2^n}) \in B(H) | f\in C(\mathbb{C}^*) \bigwedge f(1)=0 \}$. Above sequence induces the following sequence of $C^*$-algebras
\begin{equation}\label{u_2_seq}
C'(u) \subset C'(u_2) \subset ... \subset C'(u_{2^n}) \subset ...
\end{equation}

Direct limit of (\ref{u_2_seq}) (with respect to the category of $C^*$-algebras) naturally acts on $H$.
Let $C_0\left(\left(-2^n\pi, 2^n \pi \right)\right) \subset C_0\left(\left(-2^{n+1}\pi, 2^{n+1} \pi \right)\right)$ be the  natural inclusion. There is a following isomorphism
\begin{equation}\label{alg_approx}
C'(u)\approx C(\mathbb{C}^* \setminus\{1\}) \approx C((-\pi, \pi)).
\end{equation}
Similarly to \eqref{alg_approx} the sequence (\ref{u_2_seq}) is isomorphic to following sequence
\begin{equation}\label{u_2_seqr}
C_0\left(\left(-\pi,  \pi \right)\right) \xrightarrow{\phi_1}... C_0\left(\left(-2\pi, 2 \pi \right)\right) \xrightarrow{\phi_2} C_0\left(\left(-2^{2}\pi, 2^{2} \pi \right)\right) \xrightarrow{\phi_3} ...
\end{equation}
So direct limit of (\ref{u_2_seqr}) naturally acts on $H$. However direct limit of (\ref{u_2_seqr}) (in category of $C^*$-algebras) is naturally isomorphic to $C_0(\mathbb{R})$. So we have natural representation $\pi_u: C_0(\mathbb{R}) \to B(H)$.
Let $A_{\theta}\to B(H)$ be a faithful representation. Since  both $u, v \in A_{\theta}$ are unitary elements such that $\mathrm{sp}(u) = \mathrm{sp}(v) = \mathbb{C}^*$, above construction supplies two representations  $\pi_u: C_0(\mathbb{R}) \to B(H)$, $\pi_v: C_0(\mathbb{R}) \to B(H)$. Let $A\in B(H)$ be a subalgebra generated by operators $\pi_u(f)\pi_v(g)$, $\pi_v(g)\pi_u(f)$, ($f,g \in C_0(\mathbb{R})$). Denote by $\widetilde{A}_{\theta}$ the norm completion of $A$.
\begin{defn}
$ \widetilde{A}_{\theta}$ is said to be the {\it universal covering of the noncommutative torus $A_{\theta}$}.
\end{defn}
$\mathbb{Z}$ acts on $C_0(\mathbb{R})$ by following way $n \cdot f(\cdot) \mapsto f(\cdot + 2\pi n)$, $\forall f \in C_0(\mathbb{R}), \ n \in \mathbb{Z}$. This action induces action of $\mathbb{Z}^2$ on  $\widetilde{A}_{\theta}$ by following way
\begin{equation}
(n_1, n_2) \cdot \pi_u(f)\pi_v(g)= \pi_u(n_1 \cdot f)\pi_v(n_2 \cdot g), \ (n_1, n_2)\in \mathbb{Z}^2,
\end{equation}
\end{empt}
\begin{empt}\label{nc_torus_rigged}{\it Construction of a rigged module.}
There is the homeomorphism $S^1 \approx \mathbb{C}^* = \{z \in \mathbb{C} | \ |z| =1 \}$. Let $\mathcal{U}_1 = \{z \in \mathbb{C}^* | \mathrm{Im} \ z > -0.1\}$, $\mathcal{U}_2 = \{z \in \mathbb{C}^* | \mathrm{Im} \ z < 0.1\}$. $\mathcal{U}_1$, $\mathcal{U}_2$ are connected open subsets and they can be regarded as subsets of $S^1$ and $S^1=\mathcal{U}_1 \bigcup \mathcal{U}_2$. If $p: \mathbb{R} \to S^1$ is the universal covering projection then $p^{-1}(\mathcal{U}_i) = \bigsqcup \mathcal{V}_{ij}$ is disjoint union of subsets such that $\mathcal{V}_{ij}$  homeomorphic to $\mathcal{U}_i$ ($i=1,2$). Let  $1_{C(S^1))} = \sum_{i=1}^{2}a_i$ be partition of unity dominated by $\{\mathcal{U}_i\}$ $(i=1,2)$. So there are real valued positive functions $b_1, b_2 \in C(S^1)$ such that
\begin{equation}\label{two_squares_s1}
b_1^2 + b_2^2 = 1_{C(S^1)}
\end{equation}
$\mathbb{Z}$ acts on $\mathbb{R}$ by translations.
For any $\mathcal{U}_i$ we select a connected component $\mathcal{V}_i \subset p^{-1}(\mathcal{U}_i)$ and
\begin{equation}\label{n_comm_empt}
\mathcal{V}_i  \bigcap g \mathcal{V}_i = \emptyset \ \text{for any nontrivial} \ g \in \mathbb{Z}.
\end{equation}
 Let $\zeta_i \in C_0(\mathbb{R})$ is such that
\begin{equation}\label{ncomm_1}
\zeta_i(x)=\left\{
\begin{array}{c l}
    b_i(p(x)) & x \in \mathcal{V}_i \\
   0 & x \notin \mathcal{V}_i
\end{array}\right.
\end{equation}
From (\ref{n_comm_empt}), (\ref{ncomm_1}) it follows that
\begin{equation}\label{two_n_squares_s1}
1_{M(C_0(\mathbb{R}))}=\sum_{g \in \mathbb{Z}}^{}\sum_{i=1}^{2}g\zeta_ig\zeta_i.
\end{equation}
\begin{equation}\label{ncomm_2}
(g\zeta^*_i)\zeta_i = 0, \ \text{for any nontrivial} \ g \in G.
\end{equation}

Functions $b_i$  are defined on spectrum of $u$ and $v$ so there are elements $b_i(u)$, $b_i(v)\in A_{\theta}$, ($i=1,2$). Let us define following elements $e_1, ..., e_4 \in A_{\theta}$.
\begin{equation}\nonumber
e_1 = b_1(u)b_1(v), \ e_2 = b_1(u)b_2(v), \ e_3 = b_2(u)b_1(v), \ e_4 = b_2(u)b_2(v) .
\end{equation}
It is easy to check that
\begin{equation}\label{tor_adj}
e^*_1 = b_1(v)b_1(u), \ e^*_2 = b_2(v)b_1(u), \ e^*_3 = b_1(v)b_2(u), \ e_4 = b_2(v)b_2(u).
\end{equation}
From (\ref{two_squares_s1}), \eqref{tor_adj} it follows that
\begin{equation}\label{1_ncom}
1_{A_{\theta}} = \sum_{i=1}^{4}e^*_ie_i.
\end{equation}
Similarly define $\xi_1, ..., \xi_4 \in \widetilde{A}_{\theta}$ such that
\begin{equation}\nonumber
\xi_1 = \pi_u(\zeta_1)\pi_v(\zeta_1), \ \xi_2 = \pi_u(\zeta_1)\pi_v(\zeta_2), \\ \xi_3 = \pi_u(\zeta_2)\pi_v(\zeta_1), \ \xi_4 = \pi_u(\zeta_2)\pi_v(\zeta_2).
\end{equation}
From (\ref{ncomm_1})-(\ref{ncomm_2}) it follows that
\begin{equation*}
1_{M(\widetilde{A}_{\theta})}= \sum_{g \in \mathbb{Z}^2}\sum_{i=1}^{4}g\xi^*_i(g\xi_i).
\end{equation*}
\begin{equation*}
\xi^*_i\xi_j = e^*_ie_j, \ (i, j \in \{1,...,4\})
\end{equation*}
For any $i \in \{1,...,4\}$ and any $\eta \in \widetilde{A}_{\theta}$ there is a unique $b \in A_{\theta}$ such that $\xi_i\eta = \xi_i b$ we define $\langle \xi_i, \eta \rangle \stackrel{\text{def}}{=} e^*_ib\in A_{\theta}$, $\langle \eta, \xi_i \rangle \stackrel{\text{def}}{=} \langle \xi_i, \eta \rangle^* \in A_{\theta}$. Let  $X \subset \widetilde{A}_{\theta}$ be a $B$-module such that for any $\xi \in X$ the series
\begin{equation}\nonumber
\sum_{g \in \mathbb{Z}^2}\sum_{i=1}^{4}\langle \xi, g\xi_i \rangle \langle g\xi_i, \xi \rangle
\end{equation}
is norm convergent. We define a scalar product on $X$ such that
\begin{equation}\nonumber
\langle \xi, \zeta\rangle_X = \sum_{g\in \mathbb{Z}^2}\sum_{i=1}^{4}\langle \xi, g\xi_i \rangle \langle g\xi_i, \zeta \rangle, \ (\xi, \zeta \in X)
\end{equation}
The natural action $\mathbb{Z}^2$  on $\widetilde{A}_{\theta}$, induces the action of  $\mathbb{Z}^2$ on $X$, so $X$ is a $\mathbb{Z}_2$-equivariant $A_{\theta}$-rigged $\widetilde{A}_{\theta}$-module.
 From (\ref{ncomm_1}), (\ref{two_n_squares_s1}) (\ref{ncomm_2}) it follows that  $_{\widetilde{A}_{\theta}}X_{A_{\theta}}$ satisfied to conditions of theorem \ref{main_lem}. So  $_{\widetilde{A}_{\theta}}X_{A_{\theta}}$  is a $\mathbb{Z}^2$-Galois $A_{\theta}$-rigged $\widetilde{A}_{\theta}$-module.  $\widetilde{A}_{\theta}$ is the subordinated to $\{\xi_i\}_{i\in I}$ algebra. So we have a Galois quadruple
 \begin{equation}\label{nt_inf_quadruple}
\left(A_{\theta}, \ \widetilde{A}_{\theta}, \ _{\widetilde{A}_{\theta}}X_{A_{\theta}}, \mathbb{Z}^2\right).
 \end{equation}
  It is not known whether  action of $\mathbb{Z}^2$ on $\widetilde{A}_{\theta}$ is strictly outer.

\end{empt}

\section{Covering projections of spectral triples}
\subsection{Spectral triples}
A spectral triple can be regarded as a noncommutative generalization of a spin-manifold.
Any compact $\mathrm{spin}$-manifold corresponds to the unital spectral triple. Spectral triple axioms contain very strong condition with respect to the Dirac operator spectrum (See \ref{dim_sp_ax}). However in case of non-compact $\mathrm{spin}$-manifolds the Dirac operator spectrum is continuous \cite{nicolas_ginoux:dirac_spectrum}. So we have no a good algebraic definition of non-compact spin manifold. There are several notions of non-compact spectral triples \cite{connes_marcolli:motives}, \cite{moyal_spectral}, \cite{raimar_wulkenhaar:nc_spectral_triple},  but these notions require a deformation of the Dirac operator. It is known that any covering of a spin-manifold is also (naturally) a spin-manifold. Otherwise any infinite covering of any (locally) compact space is non-compact. So any infinite covering of spectral triple can be regarded as a non-compact noncommutative spin-manifold. This definition does not require a deformation of the Dirac operator.

\begin{defn}\cite{hajac:toknotes}
\label{df:spec-triple}
A (unital) {\bf \index{spectral triple} spectral triple} $(\mathcal{A}, H, D)$ consists of:
\begin{itemize}
\item
an \emph{algebra} $\mathcal{A}$ with an involution $a \mapsto a^*$,  equipped
with a faithful representation on:
\item
a \emph{Hilbert space} $H$; and also
\item
a \emph{selfadjoint operator} $D$ on $H$, with dense domain
$\Dom D \subset H$, such that $a(\Dom D) \subseteq \Dom D$ for all
$a \in \mathcal{A}$,
\end{itemize}
satisfying the following two conditions:
\begin{itemize}
\item
the operator $[D,a]$, defined initially on $\Dom D$, extends to a
\emph{bounded operator} on $H$, for each $a \in \mathcal{A}$;
\item
$D$ has \emph{compact resolvent}: $(D - \lambda)^{-1}$ is compact, when
$\lambda \notin \mathrm{sp}(D)$.
\end{itemize}
\end{defn}
Spectral triples should satisfy to several axioms, one of them is a dimension axiom.
\begin{empt}\label{dim_sp_ax}{\it Dimension axiom.} See \cite{varilly:noncom})

There is an integer $n$, the dimension of the spectral triple, such that the length element $ds = |D|^{-1}$ is an infinitesimal of order $1/n$. By "infinitesimal" we mean simply a compact operator on $H$. Since the days of Leibniz, an infinitesimal is conceptually a nonzero quantity smaller than any positive $\varepsilon$. Since we work on the arena of an infinite-dimensional Hilbert space, we may forgive the violation of
the requirement $T < \varepsilon$ over a finite-dimensional subspace (that may depend on $\varepsilon$). $T$ must then be an operator with discrete spectrum, with any nonzero $\lambda$ in $\mathrm{sp}(T)$ having
finite multiplicity; in other words, the operator $T$ must be compact. The singular values of $T$, i.e., the eigenvalues of the positive compact operator $|T|=\left(T^*T\right)^{1/2}$ are arranged in decreasing order: $\mu_0 \ge \mu_1 \ge \mu_2 \ge ...$ . We then say that $T$ is an {\it infinitesimal of order} $\alpha$ if
\begin{equation*}
\mu_k = \mathcal{O}\left(k^{-\alpha}\right) \ \text{as} \ k \to \infty.
\end{equation*}
Notice that infinitesimals of {\it first order} have singular values that form a {\it logarithmically
divergent series}:
\begin{equation}\nonumber
\mu_k(T) = \mathcal{O}\left(\frac{1}{k}\right) \Rightarrow \sigma_N(T):= \sum_{k < N}^{}=\mathcal{O}\left(\log N\right).
\end{equation}
The dimension axiom can then be reformulated as: "there is an integer $n$ for which the singular values of $D^{-n}$ form a logarithmically divergent series". If commutative triple corresponds to spin-manifold $M$ then $n=\mathrm{dim} \ M$.
\end{empt}

\subsection{Construction of covering projection}\label{sp_triple_cov}

Let $B$ and $A$ be $C^*$-algebras,  $_AX_B$ be a $G$-equivariant $B$-rigged $A$-module such that conditions of the theorem \ref{main_lem} are satisfied.
Suppose that there is a spectral triple $(\mathcal{B}, H, D)$ such that
\begin{itemize}
\item $\mathcal{B} \subset B$ is a pre-$C^*$-algebra which is a dense subalgebra in $B$.
\item there is a faithful representation $B \to B(H)$.
\end{itemize}
We would like to find a natural construction of the spectral triple $(\mathcal{A}, X \otimes_BH, \widetilde{D} )$ such that
 \begin{itemize}
\item $\mathcal{A} \subset A$ is a pre-$C^*$-algebra which is a dense subalgebra of $A$.
\item $\widetilde{D}gh = g\widetilde{D}h$, for any $g \in G$, $h \in \Dom \widetilde{D}$.
\end{itemize}

Let $\mathcal{B}_1 \subset B = \{b \in B \ | \ [D,b] \in B(H) \}$. There is a completely contractive representation of $\mathcal{B}_1$ (See \cite{bram:atricle})
\begin{equation*}
\pi_{1}:\mathcal{B}_{1} \rightarrow B(H\oplus H) \approx M_{2}(B(H))
\end{equation*}
\begin{equation}\label{sob1repr}
\pi_1(a) = \begin{pmatrix}a & 0\\ [D,a] &  a\end{pmatrix}.
\end{equation}
Representation \eqref{sob1repr} supplies an operator algebra structure on $\mathcal{B}_1$. Let  $\{e_i\}_{i \in I}$ be elements from theorem \ref{main_lem}, suppose that $e_i \in \mathcal{B}_1$ ($\forall i \in I$). Let  $J= G \times I$ be a finite or countable set of indices. Let $\mathcal{H}$ be a column Hilbert space (definition \ref{column_space}) such that the basis of  $\mathcal{H}$ is indexed by elements of $J$. Let $H_{\mathcal{B}_1}= \mathcal{H} \otimes \mathcal{B}_1$ be a Haagerup tensor product \cite{helemsky:qfa}, $\xi_j$ ($j \in J = G \times I$) be finite or countable set $g\xi_i$ ($g \in G, i \in I)$. Let $\mathcal{X}_1 \subset _AX_B$ be a norm closure of module generated by following sums
\begin{equation*}
\sum_{j \in J} b_j \xi_j; \ b_j \in \mathcal{B}_1
\end{equation*}
where the norm is defined by the representation \eqref{sob1repr}.
There are following inclusion $i: \mathcal{X}_1 \to H_{\mathcal{B}_1}$   and projection $p: H_{\mathcal{B}_1} \to \mathcal{X}_1$
\begin{equation*}
i (x) = \begin{pmatrix}
\langle x, \xi_{j_1} \rangle \\
\langle x, \xi_{j_2} \rangle \\
...
\end{pmatrix}
\end{equation*}
\begin{equation*}
p \begin{pmatrix}
b_{j_1} \\
b_{j_2} \\
...
\end{pmatrix} = \sum_{j \in J} b_j \xi_j.
\end{equation*}
such that $pi = \mathrm{Id}_{ \mathcal{X}_1}$ and $ip: \mathcal{H}_{\mathcal{B}_1}\to\mathcal{H}_{\mathcal{B}_1}$ is a projection.  According to Appendix A there is a Grassmannian connection
\begin{equation*}
\nabla : \mathcal{X}_1 \to  \mathcal{X}_1\otimes_{\mathcal{B}_1} \Omega^1\mathcal{B}_1.
\end{equation*}
From definition of $\xi_j$ it follows that these maps are $G$-equivariant, i.e.
\begin{equation*}
p(gx) = gp(x),
\end{equation*}
\begin{equation*}
i(gy) = gp(y),
\end{equation*}
so the connection $\nabla$ is $G$-equivariant, i.e. from
\begin{equation*}
\nabla x = \sum_{\iota} x_{\iota} \otimes b_{\iota}, \ (x, x_{\iota} \in \mathcal{X}_1, \ b_{\iota} \in \Omega^1\mathcal{B}_1)
\end{equation*}
it follows that
\begin{equation*}
\nabla gx = \sum_{\iota} gx_{\iota} \otimes b_{\iota}, \ (g \in G).
\end{equation*}
Let $\varphi: \Omega^1\mathcal{B}_1 \to B(H)$ be such that
\begin{equation*}
(a_1da_2)h = a_1[D, a_2]h; \ \forall h \in \mathrm{Dom}D.
\end{equation*}
Let denote $\nabla_{D}: \mathcal{X}_1 \to \mathcal{X}_1 \otimes_A B(H)$ such that from
\begin{equation*}
\nabla x = \sum_{\iota} x_{\iota} \otimes b_{\iota}
\end{equation*}
it follows that
\begin{equation*}
\nabla_D x = \sum_{\iota} x_{\iota} \otimes \varphi(b_{\iota}).
\end{equation*}
Let us define a map $\overline{D}:\mathcal{X}_1 \times \Dom D \to X \otimes_BH$ such that from
\begin{equation}
\nabla_{D} x = \sum_{\iota}x_{\iota}\otimes y_{\iota}.
\end{equation}
it follows that
\begin{equation}\label{dirac_defn}
\overline{D}(x,h) = \sum_{\iota} x_{\iota} \otimes y_{\iota}(h) + x \otimes Dh.
\end{equation}
This map is $\mathcal{B}_1$ balanced and defines a map
$\overline{D}: \mathcal{X}_1 \otimes_{\mathcal{B}_1} \Dom H \to X \otimes_B H$. Otherwise $ \mathcal{X}_1 \otimes_{\mathcal{B}_1} \mathrm{Dom} (D) \subset X \otimes_B H$ is a dense subspace. Let $\widetilde{D}$ be the closure \cite{bezandry_diagana:bound_unbound} of  operator $\overline{D}$. So we have following ingredients of spectral triple:
\begin{itemize}
\item A Hilbert space $X\otimes_B H$;
\item An unbounded operator $\widetilde{D}$ on $X\otimes_B H$.
\end{itemize}
Theorem \ref{main_lem} gives a $C^*$-algebra $A$ and its representation $A \to B(X \otimes_BH)$. This data set supplies a a Fr\'{e}chet subalgebra of $\mathcal{A}\subset A$ defined by following seminorms $\|a\|$, $\|[D,a]\|$ $\|[D,[D,a]]\|$ $\|[D,[D[D,a]]]\|$, ... .  $\mathcal{A}$ is a pre-$C^*$-algebra such that $(\mathcal{A}, X\otimes_B H, \widetilde{D})$ is a spectral triple. Details of the construction are described in \cite{bram:atricle}. The spectral triple $(\mathcal{A}, X\otimes_B H, \widetilde{D})$ can be regarded as an (infinite) covering of $(\mathcal{B}, H, D)$.

\subsection {Examples of covering projections}
\subsubsection{Coverings of spin manifolds}
\begin{empt} Commutative spectral triples correspond to spin-manifolds \cite{hajac:toknotes}. Any spin-manifold $M$ have a spinor bundle $\mathcal{S}$ which is a finite dimensional linear bundle over $M$. There are smooth sections $\Ga_{\mathrm{smooth}}(M, \mathcal{S})$ of the spinor bundle and the linear Dirac operator $\slashed{D}:\Ga_{\mathrm{smooth}}(M, \mathcal{S}) \to \Ga_{\mathrm{smooth}}(M, \mathcal{S})$. Since $\Ga_{\mathrm{smooth}}(M, \mathcal{S})$ is a dense subspace of $L^2(M, \mathcal{S})$ we can extend $\slashed{D}$ as an unbounded operator on $L^2(M, \mathcal{S})$. Then $(\Coo(M), L^2(M, \mathcal{S}), \slashed{D})$ is a commutative spectral triple.
\end{empt}

\begin{empt}\label{comm_sp_triples_dirac}{\it Topological construction}.
Let  $\pi: \widetilde{M} \to  \widetilde{M}/G = M$ be a covering projection of a spin-manifold $M$ with the spinor bundle $\mathcal{S}$. Let $\widetilde{\mathcal{S}}$ be the pullback of $\mathcal{S}$ by $\pi$. Dirac operator is local, i.e. for any open subset $U \subset M$ there is the restriction $\slashed{D}|_U:\Ga_{\mathrm{smooth}}(U, \mathcal{S}|_U) \to \Ga_{\mathrm{smooth}}(U, \mathcal{S}|_U)$. Dirac operator defines all its restrictions and vice versa. Let $x_0\in \widetilde{M}$ be any point and $U\subset \widetilde{M}$ is such that $x_0\in U$ and $\pi |_U: U \to \pi (U)$ is a homeomorphism. We have a natural isomorphism of vector spaces
\begin{equation}\label{spin_cov}
 \Ga_{\mathrm{smooth}}\left(\pi(U), \mathcal{S}|_{\pi(U)}\right) \approx  \Ga_{\mathrm{smooth}}\left(U, \mathcal{\widetilde{S}}|_U\right).
\end{equation}
Dirac operator is defined on $\Ga_{\mathrm{smooth}}\left(\pi(U), \mathcal{S}|_{\pi(U)}\right)$ and from \eqref{spin_cov} it follows that there is the natural Dirac operator on $\Ga_{\mathrm{smooth}}\left(U, \mathcal{\widetilde{S}}|_U\right)$. For any small open subset $U\in \widetilde{M}$ we have a restriction of the Dirac operator. These restrictions supply the global definition of the Dirac operator on the $\widetilde{M}$.
\end{empt}
\begin{empt}\label{comm_alg_constr}{\it Algebraic construction}.
Let $U$ be as in \ref{comm_sp_triples_dirac} and $\mathcal{I}=\{f \in C_0(\widetilde{M}) \ | \ f|_U = 0\}$ is a closed two sided ideal. Let us recall formula \eqref{1_mkx} and select $I_0 \subset I$, $g\in G$ such that
\begin{equation*}
1_{C_0(\widetilde{M})/\mathcal{I}} =  \sum_{i \in I_0}^{}g\xi_i\rangle \langle g\xi_i \ \mathrm{mod} \ \mathcal{I}.
\end{equation*}
Let $\mathcal{I}'=\{f \in C_0(M) \ | \ f|_{\pi(U)} = 0\}$. Then
\begin{equation*}
\sum_{i\in I_0}e^*_ie_i = 1 \ \mathrm{mod} \mathcal{I}',
\end{equation*}
and
\begin{equation*}
d\sum_{i\in I_0}e^*_ie_i = 0 \ \mathrm{mod} \mathcal{I}'.
\end{equation*}

Let  $x = \sum_{i\in I_0}e^*_i\xi_i \in \mathcal{X}_1$. Then $\nabla x = 0 \ \mathrm{mod} \ \mathcal{I}$
From \eqref{dirac_defn} it follows that
\begin{equation*}
D (x \otimes h) = x \otimes dh \ \mathrm{mod} \mathcal{I}.
\end{equation*}
This result coincides with above topological construction \ref{comm_sp_triples_dirac}.
\end{empt}

\subsubsection{Covering projection of noncommutative torus}

The noncommutative torus $A_{\theta}$ and its covering projection  $\widetilde{A}_{\theta}$ are  described in \ref{nc_torus_covering}. There is a dense pre-$C^*$-algebra $\mathcal{A}_{\theta} \subset A_{\theta}$ such that there is a spectral triple $(\mathcal{A}_{\theta}, H, D)$. We would like construct a dense pre-$C^*$-algebra  $\widetilde{\mathcal{A}}_{\theta} \subset \widetilde{A}_{\theta}$ and spectral triple $(\widetilde{\mathcal{A}}_{\theta}, \widetilde{H}, \widetilde{D})$ which is a covering projection of $(\mathcal{A}_{\theta}, H, D)$.
\begin{empt}\label{nt_spectral_triple}{\it Noncommutative torus as a spectral triple}.
\newline
Let
\begin{equation}\nonumber
\mathcal{A}_{\theta} = \left\{a \in A_{\theta} \ | \ a = \sum_{r,s} a_{rs}u^r, v^s \ \bigwedge \ \mathrm{sup}_{r, s \in \mathbb{Z}}\left(1 + r^2 + s^2\right)^k|a_{rs}| < \infty, \ \forall k \in \mathbb{N}\right\}.
\end{equation}
There is a linear functional $\tau_0: A_{\theta} \to \mathbb{C}$ given by
\begin{equation*}
\tau_0\left(\sum_{r,s} a_{rs}u^r v^s\right) = a_{00}.
\end{equation*}
The GNS representation space $H_0 = L^2(A_{\theta}; \tau_0)$ may be described as the completion of the vector space $A_{\theta}$ in the Hilbert norm
\begin{equation*}
\|a\|_2= \tau_0\left(\sqrt{a^*a}\right).
\end{equation*}
There are two *-homomorphisms $\pi_u: C(S^1) \to A_{\theta}$, $\pi_v: C(S^1) \to  A_{\theta}$ given by
\begin{equation}\label{circular_u}
\pi_u \left(\sum_{n \in \mathbb{Z}}a_ne^{i \varphi}\right) = \sum_{n \in \mathbb{Z}}u^n,
\end{equation}
\begin{equation}\label{circular_v}
\pi_v \left(\sum_{n \in \mathbb{Z}}a_ne^{i \varphi}\right) = \sum_{n \in \mathbb{Z}}v^n.
\end{equation}
where $\varphi$ is an angular argument of the circle.
Denote by $\underline{a}$ the image in $H_0$ of $a \in A_{\theta}$. Since $\underline{1}$ is cyclic and separating the Tomita involution is given by
\begin{equation*}
J_0(\underline{a})=\underline{a^*}.
\end{equation*}
To define structure of spectral triple we shall introduce double GNS Hilbert space $H = H_0 \oplus H_0$ and define
\begin{equation*}
J =  \begin{pmatrix}
0 & -J_0 \\
J_0 & 0
\end{pmatrix}
\end{equation*}
\end{empt}
Thee are two derivatives $\delta_1$, $\delta_2$
\begin{equation*}
\delta_1 \left(\sum_{r, s}a_{r,s}u^rv^s\right)= \sum_{rs}2\pi i r a_{rs}u^rv^s,
\end{equation*}
\begin{equation*}
\delta_2 \left(\sum_{r, s}a_{r,s}u^rv^s\right)= \sum_{rs}2\pi i s a_{rs}u^rv^s.
\end{equation*}
which satisfy Leibniz rule, i.e.
\begin{equation*}
\delta_j(ab) = (\delta_ja)b = a(\delta_jb); \ (j=1,2; \ a, b\in \mathcal{A}_{\theta}).
\end{equation*}
From \eqref{circular_u}, \eqref{circular_v} it follows that
\begin{equation}\label{d_circ_u}
\delta_1\left(\pi_u(\phi) \pi_v(\psi)\right) = 2\pi \pi_u\left(\frac{d\phi}{d\varphi}\right)\pi_v(\psi),
\end{equation}
\begin{equation}\label{d_circ_v}
\delta_2 \left(\pi_u(\phi) \pi_v(\psi)\right) = 2\pi \pi_u(\phi)\pi_v\left(\frac{d\psi}{d\varphi}\right).
\end{equation}
There are derivations
\begin{equation}\label{deri_tau}
\partial = \partial_{\tau} = \delta_1 + \tau \delta_2; \ (\tau \in \mathbb{C}, \ \mathrm{Im}(\tau)\neq 0),
\end{equation}
\begin{equation*}
\partial^+ = -J_0\partial_{\tau}J_0.
\end{equation*}
Hilbert space of the spectral triple is a direct sum $H = H_0 \oplus H_0$, action of $\mathcal{A}$ and Dirac operator are given by
\begin{equation*}
\pi(a)= \begin{pmatrix}
a & 0 \\
0 & a
\end{pmatrix},
\end{equation*}
\begin{equation*}
D= \begin{pmatrix}
0 & \underline{\partial}^+ \\
\underline{\partial} & 0
\end{pmatrix}.
\end{equation*}
\begin{empt}{\it Construction of covering projection}.
Our construction is similar to \ref{nc_torus_covering}. Suppose that functions $b_1,b_2$ from \eqref{two_squares_s1} are differentiable, i.e. $b_1, b_2 \in C^1(S^1)$. Application of \ref{sp_triple_cov} supplies a triple $(\widetilde{\mathcal{A}}_{\theta}, \widetilde{H}, \widetilde{D})$ which is a covering of $(\mathcal{A}_{\theta}, H, D)$.
\end{empt}
\begin{empt}{\it Explicit construction}.
Our construction assumes that $\widetilde{H} = X\otimes_B H$, it is reasonable to set $\widetilde{H}_0 = X \otimes_B H_0$. There is a natural inclusion $X \to H_0$ given by
\begin{equation*}
x \mapsto \underline{x} = x \otimes \underline{1}.
\end{equation*}
$X$ is a dense in $H_0$, so $\mathcal{X}_1$ is dense in $H_0$.
\end{empt}
Similarly to linear map $A_{\theta}\to H_0$ there is a partially defined linear map $\widetilde{A}_{\theta} \to \widetilde{H}_0$, $a \mapsto \underline{a}$. Similarly to \eqref{d_circ_u}, \eqref{d_circ_v} we can define homomorphisms $\pi_u, \pi_v: C_0(\mathbb{R}) \to M(\widetilde{A}_{\theta})$ and derivations $\delta_1$, $\delta_2$  on $\mathcal{X}_1$
such that
\begin{equation}\label{d_r_u}
\delta_1\left(\underline{\pi_u(\phi) \pi_v(\psi)}\right) = 2 \pi \underline{\pi_u\left(\frac{d\phi}{dx}\right)\pi_v(\psi)},
\end{equation}
\begin{equation}\label{d_r_v}
\delta_2 \left(\underline{\pi_u(\phi) \pi_v(\psi)}\right) = 2\pi \underline{\pi_u(\phi)\pi_v\left(\frac{d\psi}{dx}\right)}.
\end{equation}
where $\phi, \psi \in C_0(\mathbb{R})$. Further construction of $\widetilde{D}$ is similar to the construction of $D$ described in \ref{nt_spectral_triple}.
\begin{rem}
It is easy to find a real spectral triple structure on $(\widetilde{\mathcal{A}}_{\theta}, \widetilde{H}, \widetilde{D})$.
\end{rem}

\section{Covering projections of foliations}

\subsection{Foliations}
\begin{empt}
{\it Geometrical issues.}
Let $V$ be a smooth manifold and $TV$ is its tangent bundle, so that for each $x \in V$ , $T_xV$ is the tangent space
of $V$ at $x$. A smooth subbundle $F$ of $TV$ is called integrable if one of the following equivalent conditions is
satisfied:
\begin{enumerate}
\item  Every $x \in V$ is contained in a submanifold $W$ of $V$ such that
\begin{equation*}
T_yW = F_y; \ \forall y \in W.
\end{equation*}
\item Every $x\in V$ is in the domain $U \to V$ of a submersion $p : U \to \mathbb{R}^q$ ($q = \mathrm{Codim} F$) with
\begin{equation*}
F_y =\mathrm{Ker}(p_*)_y; \ \forall y \in V.
\end{equation*}
\item  $C^{\infty}(F) = \{X \in C^{\infty}(TV ) \ | \ x\in F_x, \forall x\in V\}$  is a Lie algebra.
\item The ideal $J(F)$ of smooth exterior differential forms which vanish on $F$ is stable by exterior differentiation.
\end{enumerate}
A foliation of $V$ is given by an integrable subbundle $F$ of $TV$. The leaves of the foliation $(V, F)$ are the
maximal connected submanifolds $L$ of $V$ with $T_x(L) = F_x, \ \forall x \in L$, and the partition of $V$ in leaves  is characterized geometrically by its "local triviality": every point $x \in V$ has a neighborhood $U$ and
a system of local coordinates $(x^j )_{j=1,...,\mathrm{dim} V}$ so that the partition of $V = \bigcup L_{\alpha}$ in connected components of leaves
corresponds to the partition of $\mathbb{R}^{\mathrm{dim} V} =\mathbb{R}^{\mathrm{dim} F} \times \mathbb{R}^{\mathrm{Codim} F}$ in the parallel affine subspaces $\mathbb{R}^{\mathrm{dim} R} \times \mathrm{pt}$.

\begin{exm}\label{kronecker_foli_exm}
The Kronecker foliation of the 2-torus $V = \mathbb{R}^2/\mathbb{Z}^2$ given by the differential
equation $dx = \theta dy$ where $\theta \notin \mathbb{Q}$, one sees that:
\begin{enumerate}
\item Though $V$ is compact, the leaves  $L_{\alpha}$, $\alpha \in A$ can fail to be compact.
\item The space $A$ of leaves $L_{\alpha}$, $\alpha \in A$, can fail to be Hausdorff and in fact the quotient topology can be
trivial (with no non trivial open subset).
\end{enumerate}
\end{exm}
Now, given a leaf $L$ of $(V, F)$ and two points $x, y \in L$ of this leaf,
any simple path $\gamma$ from $x$ to $y$ on the leaf $L$ uniquely determines a germ $h(\gamma)$ of a diffeomorphism from a transverse neighborhood of $x$ to a transverse neighborhood of
$y$. The {\it holonomy groupoid} of a leaf $L$ is the quotient of its fundamental groupoid by the equivalence relation which identifies two paths $\gamma$ and $\gamma'$ from $x$ to $y$ (both in $L$) iff $h(\gamma)= h(\gamma')$.
The holonomy covering $\widetilde{L}$ of a leaf is the covering of $L$ associated to the normal subgroup of its fundamental group $\pi_1(L)$ given
by paths with trivial holonomy. The holonomy groupoid of the foliation is the union
$G$ of the holonomy groupoids of its leaves. Given an element $\gamma$ of G, we denote by $x = s(\gamma)$ the origin of the path $\gamma$, by $y = r(\gamma)$ its end point, and $r$ and $s$ are called the range and source maps. An element $\gamma$ of G is thus given by two points $x = s(\gamma)$ and $y = r(\gamma)$ of $V$ together with an equivalence class of smooth paths: the $\gamma(t), t \in [0,1]$ with $\gamma(0) = x$ and $\gamma(1) = y$, tangent to the bundle $F$ (i.e. with  $\gamma^{\bullet}(t) \in F_{\gamma(t)}$  identifying $\gamma_1$ and $\gamma_2$ as equivalent iff the holonomy of the path $\gamma_2 \cdot \gamma_1^{-1}$ at the point $x$ is the identity.
The graph $G$ has an obvious composition law. For $\gamma, \gamma' \in G$, the composition $\gamma \circ \gamma'$ makes sense if $s(\gamma) = r(\gamma')$. The groupoid $G$ is by construction a (not necessarily Hausdorff) manifold of dimension $\mathrm{dim}G = \mathrm{dim} V + \mathrm{dim} F$.
\end{empt}
\begin{defn}\label{red_gpd_foli}
Let $N \subset V$ be a simply connected smooth submanifold such that:
\begin{itemize}
\item $\mathrm{dim}N = \mathrm{Codim}F$.
\item $N$ everywhere transverse to foliation.
\item $N$ meets every leaf.
\end{itemize}
The {\it reduced groupoid of foliation} is given by
\begin{equation*}
G_N = \{\gamma \in G \ | \ r(\gamma) \in N, \ s(\gamma) \in N \}.
\end{equation*}
$G_N$ is a smooth manifold of dimension $\mathrm{dim}N$.
\end{defn}
\begin{empt}
{\it Algebra of reduced groupoid.}
There is the algebra $C^{\infty}_c(G_N)$  of $G_N$ given by
\begin{equation*}
f * g (\gamma) = \sum_{\gamma_1 \circ \gamma_2 = \gamma}f(\gamma_1)g(\gamma_2),
\end{equation*}
\begin{equation*}
f^*(\gamma)=\overline{f(\gamma^{-1})}.
\end{equation*}
Let $G^x_N = \{\gamma\in G_N, r(\gamma)=x \}$, and $l^2(G_N,x)$ is a space of complex valued functions on $G_N^x$ such that
\begin{equation*}
\xi \in l^2(G_N,x) \Leftrightarrow	\sum_{\gamma \in G^x_N}|\xi(\gamma)|^2 < \infty.
\end{equation*}
The $C^*$-algebra norm on $C^{\infty}_c(G_N)$ is given by the supremum of the norm $\|\pi_x(f)\|$ where for each $x \in N$, $\pi_x$ is the representation of $C^{\infty}_c(G_N)$ in $l^2(G_N,x)$ given by
\begin{equation*}
(\pi_x(f)\xi)(\gamma)=\sum_{\gamma_1 \circ \gamma_2 = \gamma}f(\gamma_1)\xi(\gamma_2); \ \xi \in l^2(G_N,x), \  \gamma \in G_N, \ s(\gamma) = x.
\end{equation*}
We shall denote this $C^*$-algebra by $C^*_{r,N}(V,F)$.
\end{empt}

\subsection{Covering projections}
\begin{empt}
Let $(V,F)$ be a foliation, $N \in V$ satisfies definition \ref{red_gpd_foli}. Suppose that $\pi: \widetilde{V} \to V$ is a topological covering projection and $G$ is its grooup of covering transformations. Denote by $\widetilde{F}$  pullback of $F$ by $\pi$ and let $\widetilde{N}$ be  a preimage of $N$. If $\widetilde{N}$ is simply connected then $(\widetilde{V}, \widetilde{F})$ is a foliation and $\widetilde{N}$ satisfies definition \ref{red_gpd_foli}.  $C^{\infty}_c(G_{\widetilde{N}})$ is a right $C^{\infty}_c(G_N)$  module, the right action is given by
\begin{equation*}
f \cdot g (\gamma) = \sum_{\gamma_1 \circ \gamma_2}f(\gamma_1) g(\pi(\gamma_2)); \ f \in C^{\infty}_c(G_{\widetilde{N}}), \ g \in C^{\infty}_c(G_{N}).
\end{equation*}
This right action induces right action of $C^*_{r,N}(V,F)$ on  $C^*_{r,\widetilde{N}}(\widetilde{V},\widetilde{F})$.
We would like to show that $C^*_{r,\widetilde{N}}(\widetilde{V},\widetilde{F})$ is an (infinite) covering projection of $C^*_{r,N}(V,F)$. Let $I$ be a finite or countable set of indices such that there is a locally finite \cite{munkres:topology}  covering  $\mathcal{U}_i \subset N$ ($i \in I$) of $N$ ($N = \bigcup_{i \in I} \mathcal{U}_i$) by connected open subsets such that $p^{-1}(\mathcal{U}_i)$ ($\forall i\in I$) is a disjoint union of naturally homeomorphic to $\mathcal{U}_i$ sets. Let $1_{M(C_0(N))} = \sum_{i\in I}^{}a_i$ be a partition of unity dominated by $\{\mathcal{U}_i\}$ (See \cite{munkres:topology}). The family $e_i = \sqrt{a_i}$  satisfies condition (\ref{1_mb}), i.e.
\begin{equation}\label{comm_p1}
1_{M(C_0(N))} = \sum_{i \in I}^{} e^*_ie_i.
\end{equation}
Suppose that $e_i\in C^{\infty}(N)$, $\forall i \in I$.
Select a connected component $\mathcal{V}_i \subset \pi^{-1}(\mathcal{U}_i)$ for any $i\in I$ and $\mathcal{V}_i \bigcap g \mathcal{V}_i = \emptyset$. Let $\xi_i \in C_0(\widetilde{N})$ is such that
\begin{equation*}
\xi_i(x)=\left\{
\begin{array}{c l}
    e_i\left(p\left(x\right)\right) & x \in \mathcal{V}_i \\
   0 & x \notin \mathcal{V}_i
\end{array}\right.
\end{equation*}
It is easy to check that
\begin{equation}\label{comm_p2}
1_{M(C_0(\widetilde{N}))} = \sum_{g\in G}^{} \sum_{i\in I}^{}g\xi^*_i g\xi_i.
\end{equation}
Elements $e_i$ (resp. $\xi_i$) can be regarded as elements of $C^{\infty}(N)$ (resp. $C^{\infty}(\widetilde{N})$) given by
\begin{equation*}
e_i(\gamma)=\left\{
\begin{array}{c l}
    e_i\left(x\right) & s(\gamma)= r(\gamma) = x \in N \\
   0 & s(\gamma) \neq r(\gamma)
\end{array}\right.
\end{equation*}
\begin{equation*}
\xi_i(\widetilde{\gamma})=\left\{
\begin{array}{c l}
    \xi_i\left(\widetilde{x}\right) & s(\widetilde{\gamma})= r(\widetilde{\gamma}) = \widetilde{x} \in \widetilde{N} \\
   0 & s(\widetilde{\gamma}) \neq r(\widetilde{\gamma})
\end{array}\right.
\end{equation*}
Henceforth  $e_i$ (resp. $\xi_i$) are elements of $C^*_{r,N}(V,F)$
(resp.  $C^*_{r,\widetilde{N}}(\widetilde{V},\widetilde{F})$).
From \eqref{comm_p1}, \eqref{comm_p2} it follows that
\begin{equation}\label{comm_pf1}
1_{M(C^*_{r,N}(V,F))} = \sum_{i \in I}^{} e^*_ie_i.
\end{equation}
\begin{equation}\label{comm_pf2}
1_{M(C^*_{r,\widetilde{N}}(\widetilde{V},\widetilde{F}))} = \sum_{g\in G}^{} \sum_{i\in I}^{}g\xi^*_i g\xi_i.
\end{equation}

For any $\xi_i$ ($i\in I$) and any $\eta \in C^*_{r,\widetilde{N}}(\widetilde{V},\widetilde{F})$ there is an unique $b \in  C^*_{r,N}(V,F)$ such that $\xi_i\eta = \xi_i b$. Denote by $\langle \xi_i, \eta \rangle \stackrel{\text{def}}{=} e^*_ib\in C^*_{r,N}(V,F)$, $\langle \eta, \xi_i \rangle \stackrel{\text{def}}{=} \langle \xi_i, \eta \rangle^*$ Let us define a subspace  $X_0 \subset C^*_{r,\widetilde{N}}(\widetilde{V},\widetilde{F})$ such that for any $\zeta \in X_0$ the series
\begin{equation*}
\sum_{g\in G}\sum_{i\in I}^{}\langle \zeta, g\xi_i \rangle \langle g\xi_i, \zeta \rangle
\end{equation*}
is norm convergent. Define scalar product $\langle \xi, \zeta\rangle$ on $X_0$ such that
\begin{equation*}
\langle \xi, \zeta\rangle = \sum_{g\in G}\sum_{i\in I}^{}\langle \xi, g\xi_i \rangle \langle g\xi_i, \zeta \rangle
\end{equation*}
There is a norm $\|\xi\| = \sqrt{\langle \xi,\xi\rangle}$, let $X$ be the norm completion of $X_0$.
Natural action $G$  on $C^*_{r,\widetilde{N}}(\widetilde{V},\widetilde{F})$, induces action of $G$ on $X$, so $X$ is
Equations (\ref{comm_pf1}),(\ref{comm_pf2}) are in fact conditions \eqref{1_mb}, \eqref{1_mkx} of the theorem \ref{main_lem}. So  $_{C^*_{r,\widetilde{N}}(\widetilde{V},\widetilde{F}))}X_{C^*_{r,N}(V,F)}$  is a $G$-Galois $C^*_{r,N}(V,F)$-rigged $C^*_{r,\widetilde{N}}(\widetilde{V},\widetilde{F}))$-module. $C^*_{r,\widetilde{N}}(\widetilde{V},\widetilde{F}))$ is the subordinated to $\{\xi_i\}_{i\in I}$ algebra
It is not known whether the  action of $G$ is strictly outer in general.

\end{empt}

\section{Covering projection of isospectral deformations}
\subsection{Isospectral deformations}
\begin{empt}
A very general construction of isospectral
deformations
of noncommutative geometries is constructed in \cite{connes_landi:isospectral}. The construction
implies in particular that any
  compact spin-manifold $M$ whose isometry group has rank
$\geq 2$ admits a
natural one-parameter isospectral deformation to noncommutative geometries
$M_\theta$.
We let $(\mathcal{A} , H , D)$ be the canonical spectral triple associated with a
compact spin-manifold $M$. We recall that $\mathcal{A} = C^\infty(M)$ is
the algebra of smooth
functions on $M$, $H= L^2(M,\mathcal{S})$ is the Hilbert space of spinors and $D$
is the Dirac operator.
Let us assume that the group $\mathrm{Isom}(M)$ of isometries of $M$ has rank
$r\geq2$.
Then, we have an inclusion
\begin{equation*}
\mathbb{T}^2 \subset \mathrm{Isom}(M) \, ,
\end{equation*}
with $\mathbb{T} = \mathbb{R} / 2 \pi \mathbb{Z}$ the usual torus, and we let $U(s) , s \in
\mathbb{T}^2$, be
the corresponding unitary operators in $H = L^2(M,S)$ so that by construction
\begin{equation*}
U(s) \, D = D \, U(s).
\end{equation*}
Also,
\begin{equation*}
U(s) \, a \, U(s)^{-1} = \alpha_s(a) \, , \, \, \, \forall \, a \in \mathcal{A} \, ,
\label{actfun}
\end{equation*}
where $\alpha_s \in \mathrm{Aut}(\mathcal{A})$ is the action by isometries on the
algebra of functions on
$M$.

\noindent
We let $p = (p_1, p_2)$ be the generator of the two-parameters group $U(s)$
so that
\begin{equation*}
U(s) = \exp(i(s_1 p_1 + s_2 p_2)) \, .
\end{equation*}
The operators $p_1$ and $p_2$ commute with $D$.
Both $p_1$ and $p_2$
have integral spectrum,
\begin{equation}
\mathrm{Spec}(p_j) \subset \mathbb{Z} \, , \, \, j = 1, 2 \, .
\end{equation}

\noindent
One defines a bigrading of the algebra of bounded operators in $H$ with the
operator $T$ declared to be of bidegree
$(n_1,n_2)$ when,
\begin{equation*}
\alpha_s(T) = \exp(i(s_1 n_1 + s_2 n_2)) \, T \, , \, \, \, \forall \, s \in
\mathbb{T}^2 \, ,
\end{equation*}
where $\alpha_s(T) = U(s) \, T \, U(s)^{-1}$ as in (\ref{actfun}). \\
Any operator $T$ of class $C^\infty$ relative to $\alpha_s$ (i. e. such that
the map $s \rightarrow \alpha_s(T) $ is of class $C^\infty$ for the
norm topology) can be uniquely
written as a doubly infinite
norm convergent sum of homogeneous elements,
\begin{equation*}
T = \sum_{n_1,n_2} \, \widehat{T}_{n_1,n_2} \, ,
\end{equation*}
with $\widehat{T}_{n_1,n_2}$ of bidegree $(n_1,n_2)$ and where the sequence
of norms $||
\widehat{T}_{n_1,n_2} ||$ is of
rapid decay in $(n_1,n_2)$.
Let $\lambda = \exp(2 \pi i \theta)$. For any operator $T$ in $H$ of
class $C^\infty$ we define
its left twist $l(T)$ by
\begin{equation}\label{l_defn}
l(T) = \sum_{n_1,n_2} \, \widehat{T}_{n_1,n_2} \, \lambda^{n_2 p_1} \, ,
\end{equation}
and its right twist $r(T)$ by
\begin{equation*}
r(T) = \sum_{n_1,n_2} \, \widehat{T}_{n_1,n_2} \, \lambda^{n_1 p_2} \, ,
\end{equation*}
Since $|\lambda | = 1$ and $p_1$, $p_2$ are self-adjoint, both series
converge in norm. \\
One has,
\begin{lem}\label{conn_landi_iso_lem}\cite{connes_landi:isospectral}
\begin{itemize}
\item[{\rm a)}] Let $x$ be a homogeneous operator of bidegree $(n_1,n_2)$
and $y$ be
a homogeneous operator of  bidegree $(n'_1,n'_2)$. Then,
\begin{equation}
l(x) \, r(y) \, - \,  r(y) \, l(x) = (x \, y \, - y \, x) \,
\lambda^{n'_1 n_2} \lambda^{n_2 p_1 + n'_1 p_2}
\end{equation}
In particular, $[l(x), r(y)] = 0$ if $[x, y] = 0$.
\item[{\rm b)}] Let $x$ and $y$ be homogeneous operators as before and
define
\begin{equation}
x * y = \lambda^{n'_1 n_2} \, x y \, ; \label{star}
\end{equation}
then $l(x) l(y) = l(x * y)$.
\end{itemize}
\end{lem}

\noindent
The product $*$ defined in (\ref{star}) extends by linearity
to an associative product on the linear space of smooth operators and could
be called a $*$-product.
One could also define a deformed `right product'. If $x$ is homogeneous of
bidegree
$(n_1,n_2)$ and $y$ is homogeneous of bidegree $(n'_1,n'_2)$ the product is
defined by
\begin{equation*}
x *_{r} y = \lambda^{n_1 n'_2} \, x y \, .
\end{equation*}
Then, along the lines of the previous lemma one shows that $r(x) r(y) = r(x
*_{r} y)$.

We can now define a new spectral triple where both $H$ and the operator
$D$ are unchanged while the
algebra $\mathcal{A}$  is modified to $l(\mathcal{A})$ . By
Lemma~{\ref{conn_landi_iso_lem}}~b) one checks that  $l(\mathcal{A})$ is still an algebra. Since $D$ is of bidegree $(0,0)$ one has,
\begin{equation*}
[D, \, l(a) ] = l([D, \, a]) \label{bound}
\end{equation*}
which is enough to check that $[D, x]$ is bounded for any $x \in l(\mathcal{A})$. There is a spectral triple $\left(l(\mathcal{A}), H, D\right)$.

\end{empt}
\begin{rem} In \cite{connes_landi:isospectral} {\it real spectral triples} are described and necessary antilinear operator $J$ is constructed. This well known construction is dropped here.

\end{rem}
\subsection{Construction of covering projections}
\begin{empt}
Let $M$ be a spin-manifold such that $\mathbb{T}^2 \subset \mathrm{Isom}(M)$, $\mathcal{A} = C^{\infty}(M)$ and $l(\mathcal{A})$ is defined by \eqref{l_defn}. Let $\tilde{M} \to M$ be a covering projection. Denote by $\mathbb{T}'^2$ a copy of $\mathbb{T}^2$ for clarity. Then there is covering $\pi_T:\mathbb{T}'^2 \to \mathbb{T}^2$ such that $\mathbb{T}'^2 \subset \mathrm{Isom}(\widetilde{M})$. If $\pi_T$ is such that
\begin{equation*}
\pi_T = \mathbb{T}'^2 \approx S^1 \times S^1 \xrightarrow{k_1 \times k_2} S^1 \times S^1 \approx  \mathbb{T}^2.
\end{equation*}
then there is $\widetilde{l}(\widetilde{\mathcal{A}})= \widetilde{l}(C^{\infty}(\widetilde{M}))$ such that
\begin{equation*}
\widetilde{l}(\widetilde{T}) = \sum_{n_1,n_2} \, \widehat{\widetilde{T}}_{n_1,n_2} \, \widetilde{\lambda}^{n_2 p_1} \
\end{equation*}
where $\widetilde{\lambda}=\lambda^{\frac{1}{k_1k_2}}$ ($\lambda$ corresponds to \eqref{l_defn}). It is clear that $\widetilde{l}(\widetilde{\mathcal{A}})$ is a right $l(\mathcal{A})$ module. Let $l(A)$, $\widetilde{l}(\widetilde{A})$ be norm completions of $\widetilde{l}(\widetilde{\mathcal{A}})$ and $l(\mathcal{A})$ respectively. Manifolds $M$ and $\widetilde{M}$ are locally compact and we can define $I$, $\{e_i\}_{i\in I}$, $\{\xi_i\}_{i\in I}$ which are already constructed in \ref{comm_exm}. Suppose that $e_i \in C^{\infty}(M)$ $\forall i \in I$. Then it is easy to check that
\begin{equation*}
1_{M(l(A))}= \sum_{i \in I} \left(l(e_i)\right)^*l(e_i),
\end{equation*}
\begin{equation*}
1_{M(\widetilde{l}(\widetilde{A})} = \sum_{g\in G}^{} \sum_{i \in I}^{}g\left(\widetilde{l}(\xi)_i\right)^* g\widetilde{l}(\xi_i)
\end{equation*}
where $G$ is the group of covering transformations.
Similarly to previous examples we can construct $G$-Galois $l(A)$-rigged $\widetilde{l}(\widetilde{A})$-module $_{\widetilde{l}(\widetilde{A})}X_{l(A)}$.  $\widetilde{l}(\widetilde{A})$ is the subordinated to $\{\xi_i\}_{i\in I}$ algebra.
\end{empt}
\begin{rem}
It is easy to find the real spectral triple structure on $(\widetilde{l}(\widetilde{A}), \widetilde{H}, \widetilde{D})$.
\end{rem}

\section{Noncommutative covering projections and the Dixmier trace}
\subsection{The Dixmier trace}
\paragraph{} The Dixmier trace is the noncommutative analogue of integral over a manifold.
\begin{empt}\cite{varilly:noncom} The algebra $\mathcal{K}$ of compact operators on a separable, infnite-dimensional Hilbert space
contains the ideal $\mathcal{L}^1$ of traceclass operators, on which  $\|T\|_1 = \mathrm{Tr}|T|$ is a norm not to
be confused with the operator norm $\|T\|$. Let $\sigma_n(T)$ be such that
\begin{equation*}
\sigma_n(T)= \sup\left\{\|TP_n\|_1 \ | \ P_n \ \text{is a projector of rank} \ n \right\}
\end{equation*}

There is a formula \cite{connes_moscovici:local_index}, coming from real
interpolation theory of Banach spaces:
\begin{equation*}
\sigma_n(T)= \left\{\inf \{\|R\|_1+n\|S\| \ | \ R,S \in \mathcal{K}, \ R + S = T \right\}.
\end{equation*}

We can think of $\sigma_n(T)$ as the trace of $|T|$ with a {\it cutoff} at the scale $n$.  This scale does not have to be an integer; for any scale $ \lambda > 0$, we can {\it define}
\begin{equation*}
\sigma_{\lambda}(T) = \inf\left\{\|R\|_1+\lambda \|S\| \ | \ R,S \in \mathcal{K}, \ R + S = T \right\}.
\end{equation*}
If $0 < \lambda \le 1$, then $\sigma_{\lambda}(T) = \lambda\|T\|$. If $\lambda = n + t$ with $0 \le t < 1$, one checks that
\begin{equation}\label{sltn}
\sigma_{\lambda}(T)=(1-t)\sigma(T)+t \sigma_{n+1}(T),
\end{equation}
so $\lambda \mapsto \sigma_{\lambda}(T)$ is a piecewise linear, increasing, concave function on $(0,1)$.
\paragraph*{} Each $\sigma_{\lambda}$ is a norm by \eqref{sltn}, and so satisfies the triangle inequality. It is proved in \cite{varilly:noncom} that for positive compact operators, there is a triangle inequality in the opposite direction:
\begin{equation}\label{sigma_uneq}
\sigma_{\lambda}(A) + \sigma_{\mu}(B) \le \sigma_{\lambda + \mu}(A + B); \ \text{if} \ A,B > 0.
\end{equation}
It suffices to check this for integral values $\lambda=m$, $\mu =n$. If $P_m$, $P_n$ are projectors of respective ranks $m$, $n$, and if $P = P_m \vee P_n$ is the projector with range $P_mH + P_nH$, then
\begin{equation*}
\|AP_m\|_1 + \|BP_n\|_1 = Tr(P_mAP_m) + Tr(P_nBP_n) \le Tr(P(A + B)P) \le \|(A + B)P\|_1,
\end{equation*}
and \eqref{sigma_uneq} follows by taking supremum over $P_m$, $P_n$. Thus we have a sandwich of norms:
\begin{equation}\label{norm_sandwich}
\sigma_{\lambda}(A + B) \le \sigma_{\lambda}(A) + \sigma_{\lambda}(B) \le \sigma_{2\lambda}(A + B) \ \text{if} \ A,B \ge 0.
\end{equation}
\end{empt}
\begin{empt}\cite{varilly:noncom} {\it The Dixmier ideal}. The {\itfirst-order infinitesimals} can now be defined precisely as the following normed ideal of compact operators:
\begin{equation*}
\mathcal{L}^{1+}=\left\{T \in \mathcal{K} \ | \ \|T\|_{1+}=\sup_{\lambda > e} \frac{\sigma_{\lambda}(T)}{\log\lambda}< \infty\right\},
\end{equation*}

that obviously includes the traceclass operators $\mathcal{L}^1$. (On the other hand, if $p > 1$ we have $\mathcal{L}^{1+}\subset\mathcal{L}^p$, where the latter is the ideal of those $T$ such that $\mathrm{Tr}|T|^p < 1$, for which $\sigma_{\lambda}(T) = O(\lambda^{1-1/p})$.) If $T \in \mathcal{L}^{1+}$, the function $\lambda \mapsto \sigma_{\lambda}(T)/\log\lambda$ is continuous and bounded on the interval
$[e,\infty)$, i.e., it lies in the $C^*$-algebra $C_b[e,\infty)$. We can then form the following Ces\`{a}ro mean of this function:
\begin{equation*}
\tau_{\lambda}(T)=\frac{1}{\log \lambda}\int_{e}^{\lambda}\frac{\sigma_u(T)}{\log u}\frac{du}{u}.
\end{equation*}

Then $\lambda \mapsto \tau_{\lambda}(T)$ lies in $C_b[e, \infty)$ also, with upper bound $\|T\|_{1+}$. From \eqref{norm_sandwich} we can derive
that
\begin{equation*}
\tau_{\lambda}(A)+\tau_{\lambda}(B)-\tau_{\lambda}(A+B)\le \left(\|A\|_{1+}+\|B\|_{1+}\right)\log 2 \frac{\log \log \lambda}{\log \lambda},
\end{equation*}
so that $\tau_{\lambda}$ is "asymptotically additive" on positive elements of $\mathcal{L}^{1+}$.
\paragraph{} We get a true additive functional in two more steps. Firstly, let $\dot \tau(A)$ be the class of $\lambda \mapsto \tau_{\lambda}(A)$ in the quotient $C^*$-algebra $\mathcal{B} = C_b[e, \infty)/C_0[e, \infty)$. Then $\dot \tau$ is an additive,
positive-homogeneous map from the positive cone of $\mathcal{L}^{1+}$ into $\mathcal{B}$, and $\dot \tau(UAU^{-1}) = \dot \tau(A)$ for
any unitary $U$; therefore it extends to a linear map $\dot \tau: \mathcal{L}^{1+} \to \mathcal{B}$ such that $\dot \tau (ST) =\dot \tau (TS)$ for $T \in \mathcal{L}^{1+}$ and any $S$.
\paragraph{} Secondly, we follow $\dot \tau$ with any state (i.e., normalized positive linear form) $\omega:\mathcal{B} \to \mathbb{C}$.
The composition is a {\it Dixmier trace}:
\begin{equation*}
\mathrm{Tr}_{\omega}(T) = \omega(\dot \tau (T)).
\end{equation*}
\end{empt}
\begin{empt}{\it The noncommutative integral.} Unfortunately, the $C^*$-algebra $\mathcal{B}$ is not separable and there is no way to {\it exhibit} any particular state. This problem can be finessed by noticing that a function $f\in C_b[e, \infty)$ has a limit $\lim_{\lambda \to \infty} f(\lambda) = c$ if and only if $\omega(f) = c$ does not
depend on $\omega$. Let us say that an operator $T\in \mathcal{L}^{1+}$ is {\it measurable} if the function $\lambda \mapsto \tau_{\lambda}(T)$ converges as $\lambda \to \infty$, in which case any $\mathrm{Tr}_{\omega}(T)$ equals its limit. We denote by $\ncint T$ the common value of the Dixmier traces:
\begin{equation*}
\ncint T = \lim_{\lambda \to \infty} \tau_{\lambda}(T) \ \text{if this limit exists}.
\end{equation*}

We call this value the {\it noncommutative integral} of $T$.
\paragraph{}Note that if $T\in \mathcal{K}$ and $\sigma_n(T)/ \log n$ converges as $n \to \infty$, then $T$ lies in $\mathcal{L}^{1+}$ and is
measurable. 
\end{empt}

\begin{exm}\label{comm_integ_exm}{\it Commutative case}.
Let $M$ be a compact spin-manifold, and let $g$ be the Riemannian metric \cite{varilly:noncom}. There is the Riemannian volume form $\Omega$ given by
\begin{equation*}
\Omega = \sqrt{\mathrm{det}g(x)}dx^1 \wedge...\wedge dx^n.
\end{equation*}
It is proven in \cite{varilly:noncom} that for any $a \in C(M)$ following equation hold

\[
    \int_{M} a \Omega=
\begin{cases}
    m!(2\pi)^m \ncint a \slashed D^{-2m} ,& \text{if} \ \mathrm{dim}M = 2m \ \text{is even},\\
    (2m+1)!!\pi^{m+1} \ncint a |\slashed D|^{-2m-1}             &  \text{if} \ \mathrm{dim}M = 2m + 1 \ \text{is odd}.
\end{cases}
\]

\end{exm}

\subsection{Dixmier trace of the operator lift}

\begin{empt} Let  $\left(A, \widetilde{A}, _{\widetilde{A}}X_A, G\right)$ be a Galois quadruple, such that $G$ is finite. Let $A \to B(H)$ be a nondegenerated representation. Let $T \in \mathcal{L}^{1+}$ be a measurable operator. From theorem \ref{main_lem} it follows that
\begin{equation}\label{hilb_sum}
X \otimes_A H = \bigoplus_{g\in G}gH
\end{equation}
\end{empt}
\begin{defn}
We say that $\widetilde{T} \in B(X \otimes_A H)$ is the {\it lift} of $T\in B(H)$ if it is a diagonal matrix $\mathrm{diag}(T, T, ..., T)$ with respect to direct sum \eqref{hilb_sum}.
\end{defn}
\begin{empt}
Indeed $\widetilde{T}$ is a sum of $|G|$ operators such that any summand is "isomorphic" to $T$. Suppose $T \in \mathcal{L}^{1+}$ is measurable, and $\widetilde{T}$ is the lift of $T$. Since Dixmier trace is additive, the lift $\widetilde{T}\in \mathcal{L}^{1+}$ is also mesurable, and
\begin{equation}\label{dix_lift}
\ncint \widetilde{T} = |G| \ncint T.
\end{equation}
\end{empt}
\begin{rem}
Constructed in \ref{sp_triple_cov} Dirac operator can be regarded as the lift of unbounded operator.
\end{rem}
\begin{exm}{\it Finite covering projection of a spin manifold}.
Let $M$ be a compact spin-manifold, and let $\Omega$ be the  Riemannian volume form. Let $p: \widetilde{M}\to M$ be a finitely listed regular covering projection, so we have a Galois triple $\left(C(M), C(\widetilde{M}), G\right)$. The volume form can be naturally lifted to the form $\widetilde{\Omega}$. Any function $a \in C(M)$ has a lift $\widetilde{a}\in C(\widetilde{M})$, such that $\widetilde{a}(x)=a(p(x))$; $\forall x \in \widetilde{M}$. Let $V \in M$ be an open set such that $p^{-1}(V)$ is a disjoint union of homeomorphic to $V$ sets, i.e.
\begin{equation*}
p^{-1}(V) = \bigsqcup_{g \in G} g \widetilde{V}
\end{equation*}
where $\widetilde{V} \subset \widetilde{M}$ is homeomorphic to $V$.
It is clear that
\begin{equation*}
\int_V a = \Omega \int_{g\widetilde{V}} \widetilde{a} \widetilde{\Omega}; \ \forall g \in G.
\end{equation*}
From above equation it follows that
\begin{equation*}
 \int_{\widetilde{M}} \widetilde{a}\widetilde{\Omega} = |G| \int_M a \Omega.
\end{equation*}
According to example \ref{comm_integ_exm} above equation is a particular case of \eqref{dix_lift}.
\end{exm}
\begin{exm}\label{nt_fin_cov}{\it A finite covering projection of a noncommutative torus}.
Let $p : A_{\theta} \to A_{\theta'}$ be a *-homomorphism such that
\begin{itemize}
\item There are $m, n, k \in \mathbb{N}$ such that $\theta'=\frac{\theta + 2 \pi k}{mn}$;
\item $A_{\theta'}$ is generated by $u',v' \in U(A_{\theta'})$ and $p$ is given by
\begin{equation*}
u \mapsto u'^m; \ v \mapsto v'^n.
\end{equation*}
\end{itemize}
Let $G = \mathbb{Z}_m \times \mathbb{Z}_n$,  and let $g_1, g_2 \in G$ be generators which naturally correspond to the direct product decomposition. Let us define action of $G$ on $A_{\theta'}$ such that
\begin{equation*}
g_1u' = e^{\frac{2 \pi i}{m}}u'; \ g_1v'=v';
\end{equation*}
\begin{equation*}
g_2u' = u'; \ g_2v' = e^{\frac{2 \pi i}{n}}v'.
\end{equation*}
It is clear that $A_{\theta}=A_{\theta'}^G$. Let us show that $A_{\theta}\approx A_{\theta'}^G\to A_{\theta'}$ is a finite noncommutative covering projection.
\paragraph{}
Let $\mathcal{U}_1 \subset S^1$, $\mathcal{U}_2 \subset S^1$, $b_1 \in C^(S^1)$, $b_2 \in C(S^1)$, $e_1,..., e_4 \in A_{\theta}$ elements described in  \ref{nc_torus_rigged}.
A group $\mathbb{Z}_l$ acts on $C(S^1)$ such that the action of generator $g\in \mathbb{Z}_l$ is given by
\begin{equation*}
g a = e^{\frac{2 \pi i}{l}}a; \ \forall a \in C(S^1).
\end{equation*}
For any $\mathcal{U}_i$ we select a connected component $\mathcal{V}_i \subset p^{-1}(\mathcal{U}_i)$ and
\begin{equation}\label{n_comm_empt1}
\mathcal{V}_i  \bigcap g \mathcal{V}_i = \emptyset \ \text{for any nontrivial} \ g \in \mathbb{Z}_l.
\end{equation}
 Let $\zeta_i \in C_0(S^1)$ is such that
\begin{equation}\label{ncomm_11}
\zeta_i(x)=\left\{
\begin{array}{c l}
    b_i(p(x)) & x \in \mathcal{V}_i \\
   0 & x \notin \mathcal{V}_i
\end{array}\right.
\end{equation}
From (\ref{n_comm_empt1}), (\ref{ncomm_11}) it follows that
\begin{equation}\label{two_n_squares_s11}
1_{C(S^1)}=\sum_{g \in \mathbb{Z}_l}^{}\sum_{i=1}^{2}g\zeta_ig\zeta_i.
\end{equation}
\begin{equation}\label{ncomm_21}
(g\zeta^*_i)\zeta_i = 0, \ \text{for any nontrivial} \ g \in G.
\end{equation}
There are natural *-homomorphisms $\pi_u, \pi_v: C(S^1)\to A_{\theta'}$ such that $\pi_u$ (resp. $\pi_v$) maps the unitary generator of $C(S^1)$ to $u$ (resp. $v$).
Similarly let $\xi_1, ..., \xi_4 \in A_{\theta'}$ be such that
\begin{equation}\nonumber
\xi_1 = \pi_u(\zeta_1)\pi_v(\zeta_1); \ \xi_2 =\pi_u(\zeta_1)\pi_v(\zeta_2); \ \xi_3 = \pi_u(\zeta_2)\pi_v(\zeta_1); \ \xi_4 = \pi_u(\zeta_2)\pi_v(\zeta_2).
\end{equation}
From (\ref{ncomm_11})-(\ref{ncomm_21}) it follows that
\begin{equation*}
1_{A_{\theta'}}= \sum_{g \in \mathbb{Z}_m \times \mathbb{Z}_n}\sum_{i=1}^{4}g\xi^*_i(g\xi_i).
\end{equation*}
So we have a Galois triple
\begin{equation}\label{nt_fin_triple}
\left(A_{\theta}, A_{\theta'}, \mathbb{Z}_m\times\mathbb{Z}_n\right).
\end{equation}

Other conditions of theorem \ref{main_lem} can be easily checked. It is shown \cite{varilly:noncom} that area of the noncommutative torus $A_{\theta}$ is given by
\begin{equation*}
2 \pi \ncint D^{-2} = \frac{1}{\mathrm{Im}\tau}
\end{equation*}
where $\tau$ defines the derivation \eqref{deri_tau}. From $|\mathbb{Z}_m\times\mathbb{Z}_n| = mn$ if follows that
\begin{equation*}
2 \pi \ncint \widetilde{D}^{-2} = mn \ 2 \pi \ncint D^{-2} = nm \frac{1}{\mathrm{Im}\tau}.
\end{equation*}

\end{exm}

\section{Generated by multipliers extensions}
\paragraph{} This section contains described in \cite{ivankov:nc_cov_k_hom} constructions. However \cite{ivankov:nc_cov_k_hom} is not concerned with infinite noncommutative covering projections. Several results concerned with infinite covering projections has been added.
\begin{defn}\label{gen_by_v_extension}
Let $A$ be a $C^*$-algebra, $A\rightarrow B(H)$ is a faithful representation, and $u \in U(A^+)$, $v \in U(B(H))$, be such that $v^n=u$ and $v^i \notin U(A^+)$, ($i=1,..., n-1$). A {\it generated by $v$ extension} is a minimal subalgebra of $B(H)$ which contains following operators:
\begin{enumerate}
\item $v^i a; \ (a \in A, \ i=0, ..., n-1)$
\item $a v^i$.
\end{enumerate}
Denote by $A\{v\}$ a generated by $v$ extension. It is clear that $v$ is a multiplier of $A\{v\}$.
Number $n$ is said to be the {\it degree} of the extension.
\end{defn}
\begin{rem}
Sometimes a triple $\left(A, \ A\{v\}, \ \mathbb{Z}_n\right)$ is a noncommutative covering projection but it is not always true, see for example \ref{s_3_covering_sample}.
\end{rem}
\begin{lem}\label{fin_morita_lem} Let $A$ be a $C^*$-algebra, $A\rightarrow B(H)$ is a faithful representation, let  $u\in U(A^+)$ be an unitary element such that $\mathrm{sp}(u)=\mathbb{C}^*=\{z \in \mathbb{C} \ | \ |z| = 1 \}$, $\xi, \eta \in B_{\infty}(\mathrm{sp}(u))$ are Borel-measurable functions such that $\xi^n = \eta^n = \mathrm{Id}_{\mathbb{C}^*}$. Then there is an isomorphism
\begin{equation}
A\{\xi(u)\} \otimes \mathcal{K} \rightarrow A\{\eta(u)\} \otimes \mathcal{K}
\end{equation}
which is a left $A$-module isomorphism. The isomorphism is given by
\begin{equation}
\xi(u) \otimes x \mapsto \eta(u) \otimes \xi\eta^{-1}(u)x; \ (x \in \mathcal{K}).
\end{equation}
\end{lem}
\begin{proof}
Follows from the equality $\xi(u)= \xi\eta^{-1}(\eta(u))$.
\end{proof}

\begin{defn}\label{root_n_defn}
A {\it $n^{\mathrm{th}}$ root of identity map} is a Borel-measurable function $\phi \in \ B_{\infty}(\mathbb{C}^*)$  such that
\begin{equation}\label{root_n_eqn}
\phi^n = \mathrm{Id}_{\mathbb{C}^*}.
\end{equation}
 \end{defn}
 \begin{lem}\label{full_sp_lem}
 Let $A$ be a $C^*$-algebra, $u \in U\left(\left(A \otimes \mathcal{K}\right)^+\right)$ is such that $[u]\neq 0 \in K_1(A)$ then $\mathrm{sp}(u)= \mathbb{C}^*=\{z \in \mathbb{C} \ | \ |z| = 1\}$.
 \end{lem}
 \begin{proof}
  $\mathrm{sp}(u) \subset \mathbb{C}^*$ since $u$ is an unitary. Suppose $z_0 \in \mathbb{C}$ be such that  $z_0 \notin \mathrm{sp}(u)$ and $z_1 = -z_0$. Let $\varphi: \mathrm{sp}(u) \times [0,1] \to \mathbb{C}^*$ be such that
  \begin{equation*}
  \varphi(z_1e^{i \phi}, t)= z_1 e^{i (1-t)\phi}; \ \phi \in (-\pi, \pi), \ t \in [0,1].
  \end{equation*}
  There is a homotopy $u_t = \varphi(u, t)\in  U((A \otimes \mathcal{K})^+)$ such that $u_0 = u$, $u_1 = z_1$. From $[z_1] = 0 \in K_1(A)$ it follows that $[u]=0\in K_1(A)$. So there is a contradiction which proves this lemma.

 \end{proof}

 \begin{lem}\label{fin_unitary_ext}
 Let $A \to A\{v\}$ be a generated by $v$ extension of degree $n$, such that $v^n = u \in U(A)$. There is the natural action of $\mathbb{Z}_n$ on $A\{v\}$ such that the action of  $\mathbb{Z}_n$ generator is given by $v \mapsto e^{\frac{2\pi i}{n}}v$, and $\left(A, \ A\{v\}, \ \mathbb{Z}_n\right)$ is a Galois triple.
 \end{lem}
 \begin{proof}
  Let $\pi_S : \widetilde{S}^1 \to S^1$ be the $n$-listed covering projection of the circle. The tilde notation  $\widetilde{S}^1$ of the circle is used for clarity. It is well known that $G\left(\widetilde{S}^1 \ | \ S^1\right) = \mathbb{Z}_n$. From \ref{comm_exm} it follows that $\left(C(S^1), \ (C(\widetilde{S}^1), \ \mathbb{Z}_n \right)$ is a Galois triple. From the theorem \ref{main_lem} it follows that there is a set $I$, an sets  $\{e^S_i\}_{i\in I} \subset C(S^1)$, $\{\xi^S_i\}_{i\in I} \subset C(\widetilde{S}^1)$ such that
 \begin{enumerate}
 \item
 \begin{equation*}
 1_{C(S^1)} =  \sum_{i\in I}^{}e^{S*}_ie^S_i,
 \end{equation*}
 \item
 \begin{equation*}
 1_{C(\widetilde{S}^1)} = \sum_{g\in \mathbb{Z}_n}^{} \sum_{i \in I}^{}g\xi^S_i\rangle \langle g\xi^S_i ,
 \end{equation*}
 \item
 \begin{equation*}
 \langle \xi^S_i, \xi^S_i \rangle_{C(S^1)} = e_i^{S*}e^S_i,
 \end{equation*}
 \item
 \begin{equation*}
 \langle g\xi^S_i, \xi^S_i\rangle_{C(S^1)}=0, \ \text{for any nontrivial} \ g \in \mathbb{Z}_n.
 \end{equation*}
 Let $u^S \in U(C(S^1))$ and $\widetilde{u}^S \in U(C(\widetilde{S}^1))$ be natural unitary generators of $C(S^1)$ and $C(\widetilde{S}^1)$ respectively. Then $e^S_i$ and $\xi^S_i$ are continuous functions of $u^S$ and $\widetilde{u}^S$ respectively, i.e. $e^S_i = e^S_i(u^S)$ and $\xi^S_i=\xi^S_i(\widetilde{u}^S)$. Let  $e_i \in A$, $\xi_i \in A\{v\}$ be given by
 \begin{equation*}
e_i = e^S_i(u),
 \end{equation*}
 \begin{equation*}
\xi_i = \xi_i^S(v).
 \end{equation*}
  From previous conditions of this proof it follows that $e_i$, $\xi_i$ satisfy conditions of the theorem \ref{main_lem}.
 \end{enumerate}
 \end{proof}
 \begin{exm}
A Galois triple $\left(A_{\theta}, A_{\theta'}, \mathbb{Z}_m\times\mathbb{Z}_n\right)$ given by \eqref{nt_fin_triple} can be constructed as composition of extensions generated by unitary elements $u'$, $v'$ such that $u'^m=u$, $v'^n=v$, i.e. $A_{\theta'}=\left(A_{\theta}\{u'\}\right)\{v'\}$.
 \end{exm}
\begin{empt}\label{inf_unitary}
Let $A$ be a $C^*$-algebra, and let $u \in U(A)$ be such that
\begin{equation*}
\mathrm{sp}(u)= \mathbb{C}^*=\{z \in \mathbb{C} \ | \ |z| = 1\},
\end{equation*}
\begin{equation}\label{inf_cond}
u \neq v^n, \ \forall v \in A, n \neq 1 \in \mathbb{N}.
\end{equation}
Let $A \to B(H)$ be a faithful representation. Similarly to construction from \ref{inf_cov_nct} we can construct a representation $\pi_u: C_0(\mathbb{R}):\to B(H)$. Let $\pi_{R}: \mathbb{R} \to S^1$ be a well known covering projection, it is known that $G\left(\mathbb{R} \ | \ S^1\right)=G\left(C_0(\mathbb{R}) \ | \ C(S^1)\right)= \mathbb{Z}$. Let $X_0 \subset B(H)$ be an involutive algebra generated by operators $\pi_u(x)$; ($x\in C_0(\mathbb{R})$) and $a \in A$. From $G\left(C_0(\mathbb{R}) \ | \ C(S^1)\right)= \mathbb{Z}$ it follows that $\mathbb{Z}$ naturally acts on $C_0(\mathbb{R})$. So there is a natural action of $\mathbb{Z}$ on $\pi_u(C_0(\mathbb{R}))$, and therefore $\mathbb{Z}$ naturally acts on $X_1$. Let $X_1 \subset X_0$ be such that
\begin{equation*}
X_1 = \left\{x \in  X_0 \ | \ \exists a\in A; \ \sum_{g \in \mathbb{Z}}g(x^*x) =a  \right\}.
\end{equation*}
where a sum of the series means the strict convergence.
The $X_1$ space is a pre-$A$-rigged space such that $A$-valued form is given by
\begin{equation*}
\langle x, y \rangle_{A} = \sum_{g \in \mathbb{Z}} g(x^*y)\in A
\end{equation*}
where a sum of the series means the strict convergence. Let $X_A$ be the norm completion of $X_1$, so $X_A$ is an $A$-rigged space. Let us show that $\left(A, \widetilde{A}, _{\widetilde{A}}X_A, \mathbb{Z}\right)$ is a Galois quadruple, where $\widetilde{A}$ is the subordinated algebra. Similarly we can define $_{C_0(\mathbb{R})}X_{C(S^1)} \subset C_0(\mathbb{R})$ which is a $C(S^1)$-rigged space. From \ref{comm_exm} it follows that there are sets $\{e^S_i\}_{i\in I} \subset C(S^1)$, $\{\xi^R_i\}_{i\in I} \subset C_0(\mathbb{R})$ such that
 \begin{enumerate}
 \item
 \begin{equation*}
 1_{C(S^1)} =  \sum_{i\in I}^{}e^{S*}_ie^S_i,
 \end{equation*}
 \item
 \begin{equation*}
 1_{M(C(\mathbb{R}))} = \sum_{g\in \mathbb{Z}}^{} \sum_{i \in I}^{}g\xi^R_i\rangle \langle g\xi^R_i ,
 \end{equation*}
 \item
 \begin{equation*}
 \langle \xi^R_i, \xi^R_i \rangle_{C(S^1)} = e_i^{S*}e^S_i,
 \end{equation*}
\item
 \begin{equation*}
 \langle g\xi^R_i, \xi^R_i\rangle_{C(S^1)}=0, \ \text{for any nontrivial} \ g \in \mathbb{Z}.
 \end{equation*}
  \end{enumerate}
Similarly to lemma \ref{fin_unitary_ext} we can define elements $e_i = e_i^S(u)\in A$, $\xi_i = \pi_u(\xi_i^R) \in X_A$ such that these elements satisfy to conditions of the theorem \ref{main_lem}. We suppose that $\widetilde{A}$ is a subordinated to $\{\xi_i\}_{i\in I}$ algebra. So $\left(A, \widetilde{A}, _{\widetilde{A}}X_A, \mathbb{Z}\right)$ is a Galois quadruple.
\end{empt}
\begin{defn}
The Galois quadruple $\left(A, \widetilde{A}, _{\widetilde{A}}X_A, \mathbb{Z}\right)$ described in the lemma \ref{inf_unitary} is said to be an {\it infinite extension associated with unitary} $u$. Let denote $A\{\{u\}\}=\widetilde{A}$.
\end{defn}

\begin{rem}
For any $f\in C_0(\mathbb{R})$ operator $\pi_u(f)$ is a multiplier of $A\{\{u\}\}$.
\end{rem}
\begin{rem}
An infinite extension associated with unitary is non-unique because a sequence \eqref{root_sequence} is non-unique. However this construction us unique is case of stale $C^*$-algebras, see lemma \ref{infin_morita_lem}.
\end{rem}
\begin{defn}\cite{blackadar:ko}
Let $A$ be a $C^*$-algebra. The algebra $A \otimes \mathcal{K}$ is said to be the {\it stable algebra} of $A$.
\end{defn}
\begin{rem}
From $\mathcal{K} \otimes \mathcal{K}\approx\mathcal{K}$ it  follows that $A \otimes \mathcal{K}\otimes\mathcal{K} \approx A \otimes \mathcal{K}$,  i.e. if $A$ is a stable $C^*$-algebra then $A \approx A \otimes\mathcal{K}$. If $X$ is a rigged $A$-space then $X \otimes \mathcal{K}$ is a $A \otimes \mathcal{K}$-rigged space.
\end{rem}

\begin{lem}\label{infin_morita_lem}
Let  $\left(A \otimes \mathcal{K}, \widetilde{A}' \otimes \mathcal{K}, X' \otimes \mathcal{K}, \mathbb{Z}\right)$,  $\left(A \otimes \mathcal{K}, \widetilde{A}'' \otimes \mathcal{K}, X'' \otimes \mathcal{K}, \mathbb{Z}\right)$ be infinite extensions associated with unitary $u \in M(A\otimes \mathcal{K})$. Then there is a natural isomorphism $X'\otimes \mathcal{K} \approx X'' \otimes\mathcal{K}$ of $A\otimes \mathcal{K}$-rigged spaces. From this isomorphism it follows the *-isomophism $\widetilde{A}' \otimes \mathcal{K} \approx \widetilde{A}'' \otimes \mathcal{K}$.
\end{lem}
\begin{proof}
Construction \ref{inf_unitary} implies representations $A \to B(H)$ and $\pi_u:C_0(\mathbb{R})\to B(H)$, where $\pi_u$ is associated with a sequence $\{\phi_i\}_{i\in \mathbb{N}}$ be an arbitrary sequence of
 (non-unique) Borel-measurable functions $\{\phi_i \}_{i \in \mathbb{N}}$ which satisfies \eqref{root_sequence}. Let $\{\phi'_i\}$ (resp. $\{\phi''_i\}$ ) be Borel-measurable functions which corresponds to a representation $\pi'_u:C_0(\mathbb{R})\to B(H)$ (resp. $\pi''_u:C_0(\mathbb{R})\to B(H)$). Suppose that $\pi'_u$ and $\pi''_u$ are used for construction of $X'$ and $X''$ respectively. Let us construct the isomorphism $\varphi: \mathrm{Im}\pi'_u|_{C_c(\mathbb{R})} \otimes \mathcal{K} \approx \mathrm{Im}\pi''_u|_{C_c(\mathbb{R})} \otimes \mathcal{K}$ such that $\pi''_u|_{C_c(\mathbb{R})} \otimes \mathrm{Id}_{\mathcal{K}}= \varphi \circ \pi'_u|_{C_c(\mathbb{R})} \otimes \mathrm{Id}_{\mathcal{K}}$ and $\pi'_u|_{C_c(\mathbb{R})} \otimes \mathrm{Id}_{\mathcal{K}}= \varphi^{-1} \circ \pi''_u|_{C_c(\mathbb{R})} \otimes \mathrm{Id}_{\mathcal{K}}$, where $\otimes$ means an algebraic tensor product. Let $f \in C_c(\mathbb{R})$ be such that $\mathrm{supp}(f)\subset (-2^n\pi, 2^n\pi) $. Then $\pi'_u(f)=\overline{f}(\phi'_{n-1}\circ ... \circ \phi'_1(u))$, $\pi''_u(f)=\overline{f}(\phi''_{n-1}\circ ... \circ \phi''_1(u))$, where $\overline{f}\in C(S^1)$ is given by $\overline{f}(\theta)=f(\theta/2^n)$ ($\theta \in (-\pi,\pi)$). The isomorphism  $\varphi$ is given by
 \begin{equation*}
\overline{f}\left(\phi'_{n-1}\circ ... \circ \phi'_1(u)\right) \otimes k \mapsto \overline{f}\left( \phi''_{n-1}\circ ... \circ \phi''_1(u) \right)\otimes Tk
 \end{equation*}
 where $T = \overline{f}\left((\phi'_{n-1}\circ ... \circ \phi'_1 ) \left(\phi''_{n-1}\circ ... \circ \phi''_1\right)^{-1}(u)\right)\in B(H)$ and $k \in \mathcal{K}$.
  The map $\varphi$ can be extended to the unique isomorphism $\overline{\varphi}: \mathrm{Im}\pi'_u \otimes \mathcal{K} \approx \mathrm{Im}\pi''_u \otimes \mathcal{K}$, because $C_c(\mathbb{R})\subset C_0(\mathbb{R})$ is a dense subalgebra. The isomorphism $\overline{\varphi}$ gives isomorphisms of all objects in construction \ref{inf_cov_nct}. In particular we have $X' \otimes \mathcal{K} \approx X'' \otimes \mathcal{K}$.
\end{proof}
\begin{rem}
Lemma \ref{infin_morita_lem} is the infinite analogue of the lemma \ref{fin_morita_lem}.
\end{rem}
\begin{exm}
A Galois quadruple $\left(A_{\theta}, \ \widetilde{A}_{\theta}, \ _{\widetilde{A}_{\theta}}X_{A_{\theta}}, \mathbb{Z}^2\right)$ given by \eqref{nt_inf_quadruple} is a particular case of this construction, i.e. $\widetilde{A}_{\theta}\approx \left(A_{\theta}\{\{u\}\}\right)\{\{v\}\}$.
\end{exm}
\section{Noncommutative covering projections and $K$-homology}
\paragraph{} This section contains described in \cite{ivankov:nc_cov_k_hom} constructions. Some results concerned with infinite covering projections has been added.
\begin{empt}It is known that $K_1(S^1)\approx \mathbb{Z}$. If $x$ is a generator of $K(S^1)$ then for any topological space $\mathcal{X}$ there is a natural homomorphism $\varphi_{K}:\pi_1(\mathcal{X})\rightarrow K_1(\mathcal{X})$ given by
\begin{equation}\label{pi1_to_K1_hom}
[f] \mapsto K_1(f)(x)
\end{equation}
where $f$ is a representative of $[f]\in \pi_1(\mathcal{X})$. This homomorphism does not depend on a basepoint because $K_1(\mathcal{X})$ is an abelian group . So the basepoint is omitted. Let $K_{11}(\mathcal{X})\subset K_1(\mathcal{X})$ be the image of $\varphi_K$. Then $K_{11}(\mathcal{X})$ is a homotopical invariant.
\begin{exm}\label{circle_cov}
We have a natural isomorphism $\varphi_{K}:\pi_1(S^1) \to K_1(S^1)$. From $\pi_1(S^1)=\mathbb{Z}$ it follows that there is a infinite covering projection $\mathbb{R} \rightarrow S^1$.
\end{exm}
\begin{exm}\label{n_cirle_cov}
Let  $f: S^1 \rightarrow S^1$ be an $n$-listed covering projection, $C_{f}$ is the mapping cone \cite{spanier:at} of $f$. Then $\pi_1(C_{f})\approx K_1(C_f) \approx\mathbb{Z}_n$ and there is a natural isomorphism $\varphi_{K}: \pi_1(C_{f}) \to K_1(C_{f})$. There is $n$ - listed universal covering projection $f_n: \widetilde{C_{f}}\rightarrow C_{f}$.
\end{exm}\label{n_cirle_ex}
Finitely listed covering projections depend of fundamental group. Any epimorphism $\pi_1(\mathcal{X}) \rightarrow \mathbb{Z}$ (resp. $\pi_1(\mathcal{X}) \rightarrow \mathbb{Z}_n$) corresponds to the infinite sequence of finitely listed covering projections (resp. an $n$ - listed covering projection). If $\varphi : \pi_1(\mathcal{X}) \rightarrow G$ is an epimorphism ($G\approx \mathbb{Z}$ or $G\approx \mathbb{Z}_n$) such that $ \mathrm{ker} \ \varphi_K\subset \mathrm{ker} \ \varphi$ then there is an algebraic construction of these covering projections which is described in this article.
\end{empt}
\subsection{Universal coefficient theorem}

\paragraph{} The universal coefficient theorem \cite{blackadar:ko} states (in particular) a relationship between $K$ - theory and $K$- homology. For any $C^*$-algebra $A$ there is a natural homomorphism
\begin{equation}\label{free_spec_eqn}
\gamma: KK_1(A, \mathbb{C}) \to \mathrm{Hom}(K_1(A), K_0(\mathbb{C})) \approx \mathrm{Hom}(K_1(A), \mathbb{Z})
\end{equation}
which is the adjoint of following pairing
\begin{equation}\nonumber
KK(\mathbb{C}, A) \otimes KK(A, \mathbb{C}) \to KK(\mathbb{C}, \mathbb{C}).
\end{equation}
If $\tau \in KK^1(A, \mathbb{C})$ is represented by extension
\begin{equation}\nonumber
0 \to  \mathbb{C} \to D \to A \to  0
\end{equation}
then $\gamma$ is given as connecting maps $\partial$ in the associated six-term exact sequence of $K$-theory
\break

\begin{tikzpicture}\label{six_term}
  \matrix (m) [matrix of math nodes,row sep=3em,column sep=4em,minimum width=2em]
  {
     K_0(\mathbb{C}) & K_0(D) & K_0(A) \\
    K_1(\mathbb{C}) & K_1(D) & K_1(A) \\};
  \path[-stealth]
    (m-1-1) edge node [left] {} (m-1-2)
    (m-1-2) edge node [left] {} (m-1-3)
    (m-2-2) edge node [left] {} (m-2-1)
    (m-2-3) edge node [left] {} (m-2-2)
    (m-2-1) edge node [left] {$\partial$} (m-1-1)
    (m-1-3) edge node [left] {$\partial$} (m-2-3);.
\end{tikzpicture}


If $\gamma(\tau)=0$ for an extension $\tau$ then the six-term $K$-theory exact sequence degenerates into two short exact sequences
\begin{equation}\nonumber
0 \to K_i(A) \to K_i(D) \to K_i(\mathbb{C}) \to 0; \ (i=0,1)
\end{equation}
and thus determines an element $\kappa(\tau)\in \mathrm{Ext}^1(K_*(A), K_*(\mathbb{C})$.
In result we have a sequence of abelian group homomorphisms
\begin{equation}\nonumber
\mathrm{Ext}^1(K_0(A), K_0(\mathbb{C})) \to KK^1(A, \mathbb{C}) \to \mathrm{Hom}(K_1(A), K_0(\mathbb{C}))
\end{equation}
such that composition of the homomorphisms is trivial. Above sequence can be rewritten by following way
\begin{equation}\label{uct_c}
\mathrm{Ext}^1(K_0(A), \mathbb{Z}) \to K^1(A) \to \mathrm{Hom}(K_1(A), \mathbb{Z})).
\end{equation}
If $G$ is an abelian group that
\begin{equation}\nonumber
\mathrm{Ext}^1(G, \mathbb{Z}) = \mathrm{Ext}^1(G_{tors}, \mathbb{Z}),
\end{equation}
\begin{equation}\nonumber
 \mathrm{Hom}(G, \mathbb{Z}) =  \mathrm{Hom}(G / G_{tors}, \mathbb{Z})).
\end{equation}
From (\ref{uct_c}) it follows that $K^1(A)$ depends on $K_0(A)_{tors}$ and $K_1(A)/K_1(A)_{tors}$. We say that dependence(\ref{uct_c}) on $K_0(A)_{tors}$ is  a {\it free special case}  and dependence (\ref{free_spec_eqn}) is a {\it torsion special case}.

\subsection{Free special case}\label{free_spec_case}

\begin{exm}\label{circle_cov_proj}
The infinite covering projection of example \ref{circle_cov} can be constructed algebraically.  From (\ref{uct_c}) it follows that  $K_1(C(S^1))\approx \mathbb{Z}$. Let $u \in U(C(S^1))$ is such that $[u] \in K_1(S^1)$ is a generator of $K_1(S^1)$. Application of \ref{inf_unitary} gives an infinite covering projection $\left(C(S^1), \ C_0(\mathbb{R}), \ _{C_0(\mathbb{R})}X_{C(S^1)}, \ \mathbb{Z}\right)$ and $C_0(\mathbb{R})\approx C(S^1)\{\{u\}\}$.
\end{exm}
\begin{empt}\label{free_special_case_general} {\it General construction}.
Construction of example \ref{circle_cov_proj} can be generalized. Let $A$ be a $C^*$-algebra such that $K^1(A)\approx G \oplus \mathbb{Z}$. From (\ref{uct_c}) it follows that
\begin{equation}\label{k1_direct_sum}
K_1(A)=G' \oplus \mathbb{Z}[u]
\end{equation}
where $u \in U((A \otimes \mathcal{K})^+)$ satisfies condition \eqref{inf_cond}. Construction from \ref{inf_unitary} gives a Galois quadruple $\left(A, \ A\{\{u\}\}, \ _{A\{\{u\}\}}X_A, \ \mathbb{Z} \right)$. If $A$ is stable then from the lemma \ref{infin_morita_lem} it follows that the Galois quadruple is unique.

\end{empt}
\begin{rem}\label{pi1_belongs}
If $A = C_0(\mathcal{X})$ then construction supplies a covering projection if $x\in K_1(\mathcal{X})$ belongs to image of $\pi_1(\mathcal{X}) \rightarrow K_1(\mathcal{X})$.
\end{rem}
\begin{exm}
The described in subsection \ref{nc_torus_covering} covering projection of the noncommutative torus is a particular case of the free special case.
 \end{exm}
\begin{exm}\label{s_3_covering_sample}
It is known that $S^3$ is homeomorphic to $SU(2)$, and
\begin{equation*}
K_1\left(C(S^3)\right)=K_1(C(SU(2))) \approx \mathbb{Z}.
\end{equation*} 
The group $K_1(C(SU(2)))$ is generated by unitary $u \in U(C(SU(2) \otimes \mathbb{M}_2(\mathbb{C}))$ such that $u$ can be regarded as the natural inclusion map $SU(2) \subset \mathbb{M}_2(\mathbb{C})$ and $\mathrm{sp}(u)=\{z\in \mathbb{C}\ | \ |z|=1\}$. Let $A$ be a continuous trace algebra given by $A=C(SU(2)) \otimes \mathbb{M}_2(\mathbb{C})$. Let $\phi$ be a $n^{\mathrm{th}}$ - root of identity map, and $v = \phi(u)$. From lemma \ref{fin_unitary_ext} it follows that $\left(A, \ A\{v\}, \ \mathbb{Z}_n\right)$ is a Galois triple. From subsection \ref{cont_tr_exm} it follows that $A\{v\}$ is a continuous trace algebra and $\widehat{A\{v\}}\to \hat A$ is a (topological) covering projection. Since $\hat A \approx S^3$ and 
 $\pi_1(S^3)=0$ it follows that all covering projections of $\hat A$ are trivial, i.e. covering space is homeomorphic to a disconnected union of homeomorphic to $\hat A$ spaces
 \begin{equation*}
\widehat{A\{v\}} \approx \bigsqcup_{g\in \mathbb{Z}_n}\hat A.
 \end{equation*}
Above homeomorphism can be constructed explicitly. Let $\rho: A\{v\}\to \mathbb{M}_2(\mathbb{C})$ be a irreducible representation. From constriction of $u$ it follows that $\rho(u) \in SU(2)$, so $\mathrm{det}(\rho(u))=1$. From $u=v^n$ it follows that $\mathrm{det}(\rho(v))^n=1$. Let 
\begin{equation*}
\widehat{A\{v\}}_k = \left\{\rho \in  \widehat{A\{v\}} \ | \ \mathrm{det}(\rho(v)) = e^{\frac{2\pi i k}{n}} \right\}.
\end{equation*}
 Then $\widehat{A\{v\}} = \bigsqcup_{k = 0,..., n-1}\widehat{A\{v\}}_k$ and $\widehat{A\{v\}}_k \approx \hat A$ for $k=0,..., n-1$. This example is a specific case of boring example \ref{unconn_exm} because 
 \begin{equation*}
 A\{v\} \approx \bigoplus_{g \in \mathbb{Z}_n}A,
 \end{equation*}
 i.e. $A\{v\}$ is not irreducible.
\end{exm}

\subsection{Torsion special case}\label{torsion_special_case}

\begin{exm}\label{n_cov_sample}
The universal covering projection from example \ref{n_cirle_cov} can be constructed algebraically.
Let $f: S^1 \to S^1$ be a $n$ listed covering projection  of the circle, $C_f$ is the (topological) mapping cone of $f$. $C(f): C(S^1) \to C(S^1)$ is a corresponding *- homomorphism of $C^*$-algebras ($u \mapsto u^n$),
where $u \in U(C(S^1))$ is such that $[u]\in K_1(C(S^1))$ is a generator. Algebraic mapping cone \cite{blackadar:ko}  $C_{C(f)}$  of $C(f)$ corresponds to the topological space $C_f$, i.e. $C_{C(f)} \approx C(C_f)$.  Otherwise $C(C_f)$ is an algebra of continuous maps $f [0, 1)\to C(\mathbb{C}^*)$ such that
\begin{equation}\nonumber
f(0) = \sum_{k \in \mathbb{Z}} a_k u^{kn}, \ a_k \in \mathbb{C}.
\end{equation}
A map $v = (x \mapsto  u)$ ($\forall x \in [0, 1)$) is such that $v^i\notin M(C(C_f))$ ($i=1,...,n-1$), $v^n \in M(C(C_f))$. Homomorphism $C(C_f)\to C(C_f)\{v\}$ corresponds to a $n$-listed covering projection  $\widetilde{C_f}\to C_f$ from the example \ref{n_cirle_cov}, and $C(\widetilde{C_f})\approx C(C_f)\{v\}$.

\end{exm}
\begin{empt}\label{tors_special_case_general}{\it General construction}.
Above construction can be generalized. Let $A$ be a $C^*$-algebra such that $K^1(A)= G \oplus \mathbb{Z}_n$, where $G$ is an abelian group. From (\ref{uct_c}) it follows that $K_0(A) \approx G' \oplus \mathbb{Z}_n$.
Let $Q^s(A)= M(A \otimes  \mathcal{K})/ (A \otimes \mathcal{K}$) be the stable multiplier algebra of $C^*$-algebra $A$. Then from \cite{blackadar:ko} it follows that $K_1(Q^s(A))=K_0(A)$. Let $u\in U(Q^s(A))$ be such that $K_1(Q^s(A))=G' \oplus \mathbb{Z}_n[u]$. Let $\phi$ be a $n^{\mathrm{th}}$ root of identity map such that $\phi(u^n)=u$. Let $p:  M(A \otimes  \mathcal{K}) \to  M(A \otimes  \mathcal{K}) / (A \otimes  \mathcal{K})$ be natural surjective *- homomorphism to the quotient. It is known \cite{blackadar:ko} that unitary element $v \in  U(Q^s)$ can be lifted to an unitary element $v'\in U(M(A \otimes \mathcal{K}))$ (i.e. $v = p(v')$) if and only if $[v]=0\in K^1(Q^s(A))$.  From $n[u]=[u^n]=0$ it follows that there is an unitary $w \in U(M(A \otimes \mathcal{K})$) such that $p(w)=u^n$. Let $M(A \otimes  \mathcal{K})\rightarrow B(H)$ be a faithful representation, then $\phi(w)\in U(B(H))$. If $\phi(w)\in M(A \otimes  \mathcal{K}))$ then $p(\phi(w)) = u$, however it is impossible because $[u] \neq 0 \in K^1(Q^s(A))$. So $\phi(w) \notin M(A \otimes  \mathcal{K})$ and similarly  $\phi(w)^i \notin M(A \otimes  \mathcal{K})$ ($i = 1,..., n -1)$. So we have a generated by $\phi(w)$ extension $A\otimes \mathcal{K} \rightarrow (A \otimes \mathcal{K})\{\phi(w)\}$ which corresponds to a Galois triple $\left(A\otimes\mathcal{K}, \ \left(A\otimes\mathcal{K} \right)\{\phi(w)\}, \ \mathbb{Z}_n\right)$.
\end{empt}
\begin{exm}
Let $O_n$ be a Cuntz algebra \cite{blackadar:ko}, $K_0(O_n)=\mathbb{Z}_{n-1}$. If $w\in U(M(O_n\otimes \mathcal{K}))$ satisfies conditions of  construction  \ref{tors_special_case_general} and $\phi$ is a $(n-1)^{\mathrm{th}}$ root of unity then there is a Galois triple $(O_n \otimes \mathcal{K}, \ (O_n \otimes \mathcal{K})\{\phi(w)\}, \ \mathbb{Z}_{n-1})$. However it is not known whether the natural *-homomorphism $O_n \otimes \mathcal{K}\to (O_n \otimes \mathcal{K})\{\phi(w)\}$ strictly outer.
\end{exm}

\section{Appendix A. Grassmannian connection}

 Let $A$ be an algebra over $\mathbb{C}$-algebra. A module $\Omega^1A$ of noncommutative differential 1 forms is constructed as follows. Let $\overline{A^+}$ denote the quotient vector space $A^+/ \mathbb{C}$ by scalar multipliers of identity. Let
 \begin{equation*}
 \Omega^1A = A^+\otimes \overline{A^+}.
 \end{equation*}
 Let $d:A \to \Omega^1A$ be such that
 \begin{equation*}
 da = 1 \otimes \overline{a}, \ (a \in A)
 \end{equation*}
 where $\overline{a} = a + \mathbb{C} \in \overline{A^+}$.
 Let $E$ be a right $A$ module. A connection on $E$ to be an operator $\nabla :E \to E\otimes_A \Omega^1(A)$ satisfying the Leibniz rule
 \begin{equation*}
\nabla(\xi a)=\nabla(\xi)a + \xi \otimes da, \ (\xi \in E, a \in A).
 \end{equation*}

 Consider for example a free right module $V \otimes A$, where $V$ is a finite dimensional vector space,
 and identify the forms having values in $V \otimes A$ by means of the canonical
 isomorphism
 \begin{equation*}
 V \otimes \Omega^1 A = (V \otimes A) \otimes_A \Omega^1 A.
 \end{equation*}

 Then we have a canonical connection given by the operator $\nabla = 1 \otimes d$ on $V \otimes \Omega^1 A$. As another example, suppose that the right module $E$ is a direct summand
 of $V \otimes A$, and let $i : E \to V \otimes A$ and p : $V \otimes A \to E$ be the inclusion
 and projection maps. Then on $E$ we have an induced connection, called the
{\it  Grassmannian connection} \cite{cuntz_quillen:alg_ext}, which is given by the composition
\begin{equation}\nonumber
E \otimes_A \Omega^1 A \xrightarrow{i} V \otimes \Omega^1 A \xrightarrow{1 \otimes d} V \otimes \Omega^1 A \xrightarrow{p} E \otimes_A \Omega^1 A.
\end{equation}

 Thus in this notation the Grassmannian connection is
 \begin{equation}\nonumber
 \nabla = p (1 \otimes d) i.
 \end{equation}

 If $A$ is an operator algebra then Grassmanian connection can be generalized.
 \begin{defn}\label{column_space}
 Let $H$ be a Hilbert space with countable basis, $\mathcal{H} = B(\mathbb{C}, H)$. Operator space $\mathcal{H}$ is said to be a {\it column Hilbert space}.
 \end{defn}
 Let $\mathcal{H}_A = \mathcal{H} \otimes A$ be a Haagerup vector product of operator spaces \cite{helemsky:qfa}. Then we also have a canonical connection $\nabla : \mathcal{H}_A \to \mathcal{H}_A \otimes_{A} \Omega^1A$. If $E$ is a direct summand of $\mathcal{H}_A$ then we also have Grassmanian connection given by the composition
 \begin{equation*}\label{grass_connection}
 E \otimes_A \Omega^1 A \xrightarrow{i} \mathcal{H}_A \otimes \Omega^1 A \xrightarrow{1 \otimes d} \mathcal{H}_A \otimes \Omega^1 A \xrightarrow{p} E \otimes_A \Omega^1 A.
 \end{equation*}

\end{document}